\documentclass[letter,11pt]{amsart}

\usepackage[margin=1in]{geometry}
\usepackage[useregional]{datetime2}
\usepackage[english]{babel}
\usepackage[utf8]{inputenc}
\usepackage{xcolor, graphicx,changepage,xcolor, float, array, amssymb, amsmath, mathrsfs, amsfonts, hyperref, url, amsthm, tikz-cd, fancyhdr, enumerate, bbm}
\usepackage{tikz-cd}
\allowdisplaybreaks
\usepackage{amssymb}

\makeatletter
\newcommand*{\rom}[1]{\expandafter\@slowromancap\romannumeral #1@}
\makeatother

\makeatletter
\renewcommand*\env@matrix[1][*\c@MaxMatrixCols c]{%
	\hskip -\arraycolsep
	\let\@ifnextchar\new@ifnextchar
	\array{#1}}
\makeatother

\theoremstyle{definition}
\newtheorem{theorem}{Theorem}[section]
\newtheorem{lemma}[theorem]{Lemma}
\newtheorem{definition}[theorem]{Definition}

\newtheorem{question}[theorem]{Question}

\theoremstyle{remark}
\newtheorem{remark}[theorem]{Remark}

\usepackage{setspace}
\title{Asymptotic dynamics on amenable groups and van der Corput sets}

\author[]{Sohail Farhangi}
\address[Sohail Farhangi]{Department of Mathematics, University of Adam Mickiewicz, Pozna\'n, Poland}
\address[Sohail Farhangi]{Beijing Institute of Mathematical Sciences and Applications, Beijing, P.R.C. 101408.}
\email{sohail.farhangi@gmail.com}

\author[]{Robin Tucker-Drob}
\address[Robin Tucker-Drob]{Department of Mathematics, University of Florida, Gainesville, FL, USA}
\email{r.tuckerdrob@ufl.edu}

\begin{document}

\maketitle
\begin{abstract}
    We answer a question of Bergelson and Lesigne as well as a question of Fish and Skinner. 
    The first question is answered by showing that the notion of van der Corput set does not depend on the F\o lner sequence used to define it. 
    This result has been discovered independently by Sa\'ul Rodr\'iguez Mart\'in. 
    Both ours and Rodr\'iguez's proofs proceed by first establishing a converse to the Furstenberg Correspondence Principle for amenable groups, which answers the second question. 
    This involves studying the distributions of Reiter sequences over congruent sequences of tilings of the group.
    
    Lastly, we show that many of the equivalent characterizations of van der Corput sets in $\mathbb{N}$ that do not involve F\o lner sequences remain equivalent for arbitrary countably infinite groups.
\end{abstract}

\section{Introduction}

Let $G$ be a countably infinite amenable group and $\mathcal{F}=(F_n)_{n=1}^\infty$ a left-F\o lner sequence in $G$. 
A subset $V$ of $G$ is an \textbf{$\mathcal{F}$-van der Corput set ($\mathcal{F}$-vdC set)}\footnote{While our definition of $\mathcal{F}$-vdC set seems different from that of Rodr\'guez \cite{SaulsvdC}, he shows that they are equivalent.} if for any $(c_g)_{g \in G} \subseteq \mathbb{S}^1$ satisfying

\begin{equation}
    \lim_{n\rightarrow\infty}\frac{1}{|F_n|}\sum_{g \in F_n}c_{vg}\overline{c_g} = 0\text{ for all }v \in V,
    \end{equation}
    we have
    \begin{equation}
    \lim_{n\rightarrow\infty}\frac{1}{|F_n|}\sum_{g \in F_n}c_g = 0.
\end{equation}

Bergelson and Lesigne \cite[Page 44]{vanderCorputSetsInZ^d} showed that if $V \subseteq \mathbb{Z}$ is a $([1,N])_{N = 1}^\infty$-vdC set, then it is also an $\mathcal{F}$-vdC set for any F\o lner sequence $\mathcal{F}$ in $(\mathbb{Z},+)$. They then asked whether or not the converse holds. To be more precise, if $\mathcal{F}$ is a F\o lner sequence in $\mathbb{Z}$ and $V \subseteq \mathbb{Z}$ is $\mathcal{F}$-vdC, is $V$ also a $([1,N])_{N = 1}^
\infty$-vdC set? One of the main results of this paper is Theorem \ref{FvdCIsvdCForAmenableGroups}, which yields a positive answer to this question. In fact, we show that for any countably infinite amenable group $G$, and any left-F\o lner sequences $\mathcal{F}_1$ and $\mathcal{F}_2$, a set $V \subseteq G$ is $\mathcal{F}_1$-vdC if and only if it is $\mathcal{F}_2$-vdC. Below we only state a special case of Theorem \ref{FvdCIsvdCForAmenableGroups}.

\begin{theorem}
    Let $G$ be a countably infinite amenable group and let $\mathcal{F} = (F_n)_{n = 1}^\infty$ be a left-F\o lner sequence in $G$. A set $V \subseteq G$ is an $\mathcal{F}$-vdC set if and only if for any measure preserving system $(X,\mathscr{B},\mu,(\tau_g)_{g \in G})$ and any $f:X\rightarrow\mathbb{S}^1$ satisfying $\langle \tau_vf,f\rangle = 0$ for all $v \in V$, we have $\int_Xfd\mu = 0$.
\end{theorem}

We mention that we had originally proven this result for abelian groups, and extended our proof to the case of amenable groups after discussions with Sa\'ul Rodr\'iguez Mart\'in, who had also independently answered the question of Bergelson and Lesigne, in the setting of amenable groups as \cite[Theorem 1.12]{SaulsvdC}.

The other main result of this paper is Theorem \ref{RepresentationTheoremForAmenableGroups}, which can be seen as a converse to the Furstenberg Correspondence Principle. We state a special case of this result below.

\begin{theorem}\label{MainResultIntro}
    Let $G$ be a countably infinite amenable group and let $\mathcal{F} = (F_n)_{n = 1}^\infty$ be a left-F\o lner sequence. Given a measure preserving system $(X,\mathscr{B},\mu,(\tau_g)_{g \in G})$ and a $f \in L^\infty(X,\mu)$, there exists a bounded sequence of complex numbers $(c_g)_{g \in G} \subseteq \text{range}(f)$ satisfying

    \begin{alignat*}{2}
        &\lim_{n\rightarrow\infty}\frac{1}{|F_n|}\sum_{g \in F_n}c_g = \int_Xfd\mu, \lim_{n\rightarrow\infty}\frac{1}{|F_n|}\sum_{g \in F_n}c_{hg}\overline{c_g} = \langle \tau_hf,f\rangle\text{ for all }h \in G\text{, and}\\
        &\lim_{n\rightarrow\infty}\frac{1}{|F_n|}\sum_{g \in F_n}d_{h_1g,1}^{t_1}d_{h_2g,2}^{t_2}\cdots d_{h_\ell g,\ell}^{t_\ell} = \int_X\tau_{h_1}f_1^{t_1}\tau_{h_2}f_2^{t_2}\cdots \tau_{h_\ell}f_\ell^{t_\ell} d\mu,
    \end{alignat*}
    where $\ell,t_i \in \mathbb{N}$, $h_i \in G$, $(d_{g,i})_{g \in G} \in \left\{(c_g)_{g \in G},(\overline{c_g})_{g \in G}\right\}$, $f_i \in \{f,\overline{f}\}$, and $(d_{g,i})_{g \in G} = (c_g)_{g \in G}$ if and only if $f_i = f$.
\end{theorem}
Rodr\'iguez also has similar results as \cite[Theorems 1.6, 2.7]{SaulsvdC}, and in his article he discusses in detail the relationship between these results and the Furstenberg Correspondence Principle. We mention that this topic has been previously investigated by Avigad \cite{avigad2012inverting} as well as Fish and Skinner \cite{InverseOfFurstenbergsCorrespondenceForN}.
Furthermore, Fish and Skinner \cite[Question 2]{InverseOfFurstenbergsCorrespondenceForN} asked if Theorem \ref{MainResultIntro} is true when we take $G = \mathbb{Z}$ and let $f$ be the indicator of some $A \in \mathscr{B}$, so we have obtained an affirmative answer to their question.

Let us now recall the original definition of vdC sets.
\begin{definition}\label{vdCSetDefinition}
    A set $V \subseteq \mathbb{N}$ is a \textbf{van der Corput (vdC) set} if for any sequence $(x_n)_{n = 1}^\infty \subseteq [0,1]$ for which $(x_{n+v}-x_n\pmod{1})_{n = 1}^\infty$ is uniformly distributed\footnote{A sequence $(x_n)_{n = 1}^\infty \subseteq [0,1]$ is \textbf{uniformly distributed} if for any $0 \le a < b \le 1$ we have $\lim_{N\rightarrow\infty}\frac{1}{N}|\{1 \le n \le N\ |\ x_n \in (a,b)\}| = b-a$.} for all $v \in V$, we have that $(x_n)_{n = 1}^\infty$ is uniformly distributed.
\end{definition}

One of the reasons that vdC sets are of interest is because of their many equivalent reformulations. We state some of these equivalent formulations below, and in the appendix we give some more.

\begin{theorem}\label{EquivalentCharacterizationsOfvdC}
    For $V \subseteq \mathbb{N}$, the following are equivalent:
    \begin{enumerate}[(i)]
        \item\label{UniformDistributionvdC} $V$ is a vdC set.

        \item\label{boundedcomplexvdC} For any sequence $(u_n)_{n = 1}^\infty$ of complex numbers of modulus 1, if

        \begin{equation}
            \lim_{N\rightarrow\infty}\frac{1}{N}\sum_{n = 1}^Nu_{n+v}\overline{u_n} = 0\text{, for all }v \in V\text{, then }\lim_{N\rightarrow\infty}\frac{1}{N}\sum_{n = 1}^Nu_n = 0.
        \end{equation}

        \item\label{cesaroboundedcomplexvdc} For any sequence $(u_n)_{n = 1}^\infty$ of complex numbers satisfying

        \begin{equation}
            \limsup_{N\rightarrow\infty}\frac{1}{N}\sum_{n = 1}^N|u_n|^2 < \infty\text{ and }\lim_{N\rightarrow\infty}\frac{1}{N}\sum_{n = 1}^Nu_{n+v}\overline{u_n} = 0\text{ for all }v \in V,
        \end{equation}
        we have

        \begin{equation}
            \lim_{N\rightarrow\infty}\frac{1}{N}\sum_{n = 1}^Nu_n = 0.
        \end{equation}

        \item\label{cesaroboundedhilbertvdc} For any Hilbert space $\mathcal{H}$ and any sequence $(\xi_n)_{n = 1}^\infty$ of vectors in $\mathcal{H}$ satisfying

        \begin{equation}
            \limsup_{N\rightarrow\infty}\frac{1}{N}\sum_{n = 1}^N||\xi_n||^2 < \infty\text{ and }\lim_{N\rightarrow\infty}\frac{1}{N}\sum_{n = 1}^N\langle \xi_{n+v}, \xi_n\rangle = 0\text{ for all }v \in V,
        \end{equation}
        we have

        \begin{equation}
            \lim_{N\rightarrow\infty}\left|\left|\frac{1}{N}\sum_{n = 1}^N\xi_n\right|\right| = 0.
        \end{equation}

        \item\label{MPSvdC} For any measure preserving system $(X,\mathscr{B},\mu,\tau)$ and any $f \in L^2(X,\mu)$ satisfying $\langle \tau^vf,f\rangle = 0$ for all $v \in V$, we have $\int_Xfd\mu = 0$.

        \item\label{operatorialrecurrencevdc} $V$ is a \textbf{set of operatorial recurrence}, i.e., if $\mathcal{H}$ is a Hilbert space, $U:\mathcal{H}\rightarrow\mathcal{H}$ is a unitary operator, and $\xi \in \mathcal{H}$ satisfies $\langle U^v\xi,\xi\rangle = 0$ for all $v \in V$, then $P_I\xi = 0$, where $P_I:\mathcal{H}\rightarrow\mathcal{H}$ is the orthogonal projection onto the subspace of $U$-invariant vectors.

        \item\label{operatorialrecurrencevdc+Kronecker} If $\mathcal{H}$ is a Hilbert space, $U:\mathcal{H}\rightarrow\mathcal{H}$ is a unitary operator, and $\xi \in \mathcal{H}$ satisies $\langle U^v\xi,\xi\rangle = 0$ for all $v \in V$, then $P_\mathcal{K}\xi = 0$, where $P_\mathcal{K}:\mathcal{H}\rightarrow\mathcal{H}$ is the orthogonal projection onto the smallest closed subspace of $\mathcal{H}$ containing all eigenvectors of $U$.
    \end{enumerate}
\end{theorem}

The equivalence of \eqref{UniformDistributionvdC} and \eqref{cesaroboundedcomplexvdc} is implicitly alluded to in the work of Kamae and Mendes-France \cite{KMFvdC}. The equivalence of \eqref{UniformDistributionvdC}, \eqref{boundedcomplexvdC}, and \eqref{cesaroboundedcomplexvdc} was proven in the work of Ruzsa \cite{RuzsavanderCorput}. The equivalence of \eqref{UniformDistributionvdC}, \eqref{operatorialrecurrencevdc}, and \eqref{operatorialrecurrencevdc+Kronecker} is originally due to Peres \cite{BanachLimitsvdC}. The term ``operator recurrent" was introduced by Nin\v{c}evi\'c, Rabar, and Slijep\v{c}evi\'c \cite{vdCIsOperatorialRecurrent} when they independently rediscovered the equivalence of \eqref{UniformDistributionvdC} and \eqref{operatorialrecurrencevdc} (see also \cite{StationarySequencesOfVectorsvdC} for a related characterization). The equivalence of \eqref{UniformDistributionvdC} and \eqref{cesaroboundedhilbertvdc} is due to Bergelson and Lesigne \cite{vanderCorputSetsInZ^d}. The equivalence of \eqref{MPSvdC} and \eqref{operatorialrecurrencevdc} is a well-known consequence of the Gaussian measure space construction. 

In Theorem \ref{FvdCIsvdCForAmenableGroups} we show that the characterizations of vdC sets involving F\o lner sequences mostly extend to any countably infinite amenable group $G$ (see also Question \ref{MainConjecture}). In the appendix we collect other equivalent characterizations of vdC sets/sets of operatorial recurrence, and show that these equivalences still hold for any (not necessarily amenable) countably infinite group.

\textbf{Acknowledgements:} We would like to thank Mariusz Lema\'nczyk for helping discussions regarding the Foia\c{s} and Str\u{a}til\u{a} Theorem. 
We would like to thank Joel Moreira for sharing with us his unpublished notes containing a succinct presentation of many of the equivalent definitions of vdC sets in $\mathbb{N}$. 
We would also like to thank Sa\'ul Rodr\'iguez Mart\'in for fruitful discussions that began after realizing that we had been working on similar problems.
Lastly, we would like to thank the anonymous referee for his very careful reading of this paper and the numerous comments he provided that greatly improved the exposition.

SF acknowledges being supported by grant
2019/34/E/ST1/00082 for the project “Set theoretic methods in dynamics and number theory,” NCN (The
National Science Centre of Poland). 
RTD was supported by NSF grant DMS-2246684.

\section{Preliminaries}
\subsection{Notation}
We use $G$ to denote a locally compact second countable topological group with identity $e$ and left-Haar measure $\lambda$. 
Usually $G$ will be a countable discrete group, so $\lambda$ will be the counting measure and we will simply write $|F| = \lambda(F)$ for $F \subseteq G$ in this case. 
We use $\mathcal{H}$ to denote a separable Hilbert space and $\mathcal{U}(\mathcal{H})$ to denote the set of unitary operators on $\mathcal{H}$ endowed with the strong operator topology.
A representation $\pi$ of $G$ on $\mathcal{H}$ is a measurable group homomorphism $\pi:G\rightarrow\mathcal{U}(\mathcal{H})$. 
A measure preserving system (m.p.s.) $(X,\mathscr{B},\mu,(\tau_g)_{g \in G})$ is a probability space $(X,\mathscr{B},\mu)$ and a measurable action $\tau$ of $G$ on $X$, satisfying $\mu(\tau_gA) = \mu(A)$ for all $g \in G$ and $A \in \mathscr{B}$.
We again use $\tau$ to denote the Koopman representation of $\tau$ on $L^2(X,\mu)$ that is given by $\tau f = f\circ \tau$.
We let $\mathbb{S}^1 = \{z \in \mathbb{C}\ |\ |z| = 1\}$.
For $a,b \in \mathbb{C}$ and $\epsilon > 0$, we write $a \overset{\epsilon}{=} b$ to denote $|a-b| < \epsilon$. 
\subsection{Amenable groups and tilings}
Let $G$ be a countable group with identity $e$. A \textbf{(left-)F\o lner sequence} is a sequence of finite sets $(F_n)_{n = 1}^\infty$ satisfying

\begin{equation}
    \lim_{n\rightarrow\infty}\frac{|F_n\triangle gF_n|}{|F_n|} = 0\text{ for all }g \in G. 
\end{equation}
The group $G$ is \textbf{amenable} if it possesses a F\o lner sequence. We can also give an equivalent definition of amenability in terms of sequences of asymptotically invariant probability measures. A sequence of probability measures $(\nu_n)_{n = 1}^\infty$ is \textbf{(left-)asymptotically invariant}\footnote{Since our group $G$ is countable, a probability measure $\nu$ on $G$ has the form $d\nu = fd\lambda$ with $f(g) = \mu(\{g\})$, so we do not explicitly talk about the Radon-Nikodym derivative of our measures with respect to the Haar measure $\lambda$ as is usually done with non-discrete amenable groups.} if for any $k \in G$ we have

\begin{equation}
    \lim_{n\rightarrow\infty}\int_G|\nu_n(\{kg\})-\nu_n(\{g\})|d\lambda(g) = 0,
\end{equation}
and $G$ is amenable if and only if there exists an asymptotically invariant sequence of probability measures. We mention that some texts refer to asymptotically invariant sequences of probability measures as \textbf{Reiter sequences}. We note that a F\o lner sequence $(F_n)_{n = 1}^\infty$ is naturally identified with the Reiter sequence $(\nu_n)_{n = 1}^\infty$ for which $\nu_n(\{g\}) = \frac{1}{|F_n|}\mathbbm{1}_{F_n}(g)$. Given $\epsilon > 0$ and a finite $K \subseteq G$, the probability measure $\nu$ is $(K,\epsilon)$-invariant if for every $k \in K$ we have $\int_G|\nu(\{kg\})-\nu(\{g\})|d\lambda < \epsilon$, and a finite $F \subseteq G$ is $(K,\epsilon)$-invariant if $|F\triangle kF| < \epsilon|F|$ for all $k \in K$.
\begin{definition}
    A \textit{tiling} $\mathcal{T}$ of a group $G$ is determined by two objects:
    \begin{enumerate}[(1)]
        \item a finite collection $\mathcal{S}(\mathcal{T})$ of finite subsets of $G$ containing the identity $e$, called \textit{the shapes},

        \item a finite collection $\mathcal{C}(\mathcal{T}) = \{C(S)\ |\ S \in \mathcal{S}(\mathcal{T})\}$ of disjoint subsets of $G$, called \textit{center sets} (for the shapes).
    \end{enumerate}
    The tiling $\mathcal{T}$ is then the family $\{(S,c)\ |\ S \in \mathcal{S}(\mathcal{T})\ \&\ c \in C(S)\}$ provided that $\{Sc\ |\ (S,c) \in \mathcal{T}\}$ is a partition of $G$. A \textit{tile} of $\mathcal{T}$ refers to a set of the form $T = Sc$ with $(S,c) \in \mathcal{T}$, and in this case we may also write $T \in \mathcal{T}$. A sequence $(\mathcal{T}_k)_{k = 1}^\infty$ of tilings is \textit{congruent} if each tile of $\mathcal{T}_{k+1}$ is a union of tiles of $\mathcal{T}_k$, and in this case we further assume without loss of generality that $\bigcup_{S \in \mathcal{S}(\mathcal{T}_{k+1})}C(S) \subseteq \bigcup_{S \in \mathcal{S}(\mathcal{T}_k)}C(S)$.
\end{definition}

We see that any group $G$ has a trivial tiling $\mathcal{T}$ in which $\mathcal{S}(\mathcal{T}) = \{\{e\}\}$ and $\mathcal{C}(\mathcal{T}) = \{G\}$. When the group $G$ is amenable, we look for more interesting tilings by requiring that the shapes of the tiling be $(K,\epsilon)$-invariant for some finite $K \subseteq G$ and $\epsilon > 0$. We now recall a special case of a result of Downarowicz, Huczek, and Zhang regarding such tilings.
\begin{theorem}[{\cite[Theorem 5.2]{TilingAmenableGroups}}]\label{CongruentTilingTheorem}
    Let $G$ be a countably infinite amenable group. Fix a converging to zero sequence $\epsilon_k > 0$ and a sequence $K_k$ of finite subsets of $G$. There exists a congruent sequence of tilings $(\mathcal{T}_k)_{k = 1}^\infty$ of $G$ such that the shapes of $\mathcal{T}_k$ are $(K_k,\epsilon_k)$-invariant.
\end{theorem}

\begin{lemma}\label{GoodTilingLemma}
    Let $G$ be a countably infinite amenable group, let $Q \subseteq G$ be finite, and let $\epsilon > 0$ be arbitrary. Let $\mathcal{T}$ be a tiling of $G$ for which each tile is $(Q,\epsilon)$-invariant, let $M = |\mathcal{S}(\mathcal{T})|$, and let $U = \bigcup_{S \in \mathcal{S}(\mathcal{T})}S$.  
    Suppose that $\nu$ is a probability measure on $G$ that is $\left(QUU^{-1},\frac{\epsilon}{M|U|}\right)$-invariant.
    For each tile $T$ of $\mathcal{T}$ let $\nu_T$ be the measure given by $\nu_T(A) := \frac{\nu(A\cap T)}{\nu(T)}$ (with the convention that $\frac{0}{0} = 0$).
    \begin{enumerate}[(i)]
        \item For any $g \in Q$ we have
    
        \begin{equation}\label{TilingBoundEquation1}
            \sum_{T \in \mathcal{T}}\nu(gT\setminus T) < 3\epsilon\text{ and }\sum_{T \in \mathcal{T}}\nu(T\setminus g^{-1}T) < 3\epsilon.
        \end{equation}
    
        \item There exists a finite set $D$ that is a union of tiles of $\mathcal{T}$ such that $\nu(D) > 1-4\sqrt{\epsilon}$, and for each tile $T \subseteq D$, the probability measure $\nu_T$ is $(Q,\sqrt{\epsilon}|Q|)$-invariant.
    \end{enumerate}
\end{lemma}

\begin{proof}
    We begin by proving (i). Let us fix an $S \in \mathcal{S}(\mathcal{T})$ and a $g \in Q$, and let us assume that $gS\setminus S \neq \emptyset$.\footnote{Later we will take sums over sets of the form $gS\setminus S$, and the sums will be empty if $gS\setminus S$ is empty, hence they will be negligible.} 
    Since $S$ is $(Q,\epsilon)$-invariant, we have $|gS\setminus S| < \epsilon|S|$, so there exist injections $\phi_{S,1},\cdots,\phi_{S,n_S}:gS\setminus S\rightarrow S$ for which $S \subseteq \bigcup_{m = 1}^{n_S}\phi_{S,m}(gS\setminus S)$, each $s \in S$ is contained in $\phi_{S,m}(gS\setminus S)$ for at most 2 values of $m$, and $\frac{1}{n_S} < \epsilon$. 
    We see that for $x \in gS$ and $y := \phi_{S,m}(x) \in S$, we have $t := xy^{-1} \in gSS^{-1} \subseteq QUU^{-1}$, hence

    \begin{alignat*}{2}
        &\sum_{T \in \mathcal{T}}\nu(gT\setminus T) = \sum_{S \in \mathcal{S}(\mathcal{T})}\sum_{c \in C(S)}\nu(gSc\setminus Sc) = \sum_{S \in \mathcal{S}(\mathcal{T})}\sum_{c \in C(S)}\sum_{x \in gS\setminus S}\nu(\{xc\})\\
        = &\sum_{S \in \mathcal{S}(\mathcal{T})}\sum_{c \in C(S)}\frac{1}{n_S}\sum_{m = 1}^{n_S}\sum_{x \in gS\setminus S}\nu(\{xc\})\\
        \le& \sum_{S \in \mathcal{S}(\mathcal{T})}\sum_{c \in C(S)}\frac{1}{n_S}\sum_{m = 1}^{n_S}\sum_{x \in gS\setminus S}\nu(\{\phi_{S,m}(x)c\})\\
        &+\sum_{S \in \mathcal{S}(\mathcal{T})}\sum_{c \in C(S)}\frac{1}{n_S}\sum_{m = 1}^{n_S}\sum_{x \in gS\setminus S}|\nu(\{\phi_{S,m}(x)c\})-\nu(\{xc\})|\\
        \le& \sum_{S \in \mathcal{S}(\mathcal{T})}\sum_{c \in C(S)}\frac{2\nu(Sc)}{n_S}+\sum_{S \in \mathcal{S}(\mathcal{T})}\frac{1}{n_S}\sum_{m = 1}^{n_S}\sum_{x \in gS\setminus S}\sum_{c \in G}|\nu(\{\phi_{S,m}(x)c\})-\nu(\{xc\})|\\
        < &2\epsilon\sum_{S \in \mathcal{S}(\mathcal{T})}\sum_{c \in C(S)}\nu(Sc)+\sum_{S \in \mathcal{S}(\mathcal{T})}\frac{1}{n_S}\sum_{m = 1}^{n_S}\sum_{x \in gS\setminus S}\frac{\epsilon}{M|U|} < 3\epsilon.\\
    \end{alignat*}
    To prove the second claim of (i), it suffices to argue as above after replacing the maps $\phi_{S,m}$ with the maps $\phi_{S,m}':S\setminus g^{-1}S\rightarrow S$ given by $\phi_{S,m}'(x) = \phi_{S,m}(gx)$.
    
    To prove (ii), we see that for any $g \in Q$ we have
    
    \begin{alignat*}{2}
        \epsilon \ge &\int_G|\nu(\{gx\})-\nu(\{x\})|d\lambda(x) = \sum_{T \in \mathcal{T}}\int_T|\nu(\{gx\})-\nu(\{x\})|d\lambda(x) \\
        \ge &\sum_{T \in \mathcal{T}}\nu(T)\int_T|\nu_T(\{gx\})-\nu_T(\{x\})|d\lambda(x)-\sum_{T \in \mathcal{T}}\int_{T\setminus g^{-1}T}\nu(\{gx\})d\lambda(x)\\
        > & \sum_{T \in \mathcal{T}}\nu(T)\int_T|\nu_T(\{gx\})-\nu_T(\{x\})|d\lambda(x)-3\epsilon.
    \end{alignat*}
    For $g \in Q$, let $A_g$ denote the set of tiles $T$ for which either $\nu _T$ is not $(\{g\},\sqrt{\epsilon}|Q|)$-invariant or with $\nu (T) = 0$, and let $B_g$ be the set of all other tiles. We see that for $g \in Q$ we have

    \begin{alignat*}{2}
        4\epsilon > &\sum_{T \in \mathcal{T}}\nu(T)\int_T|\nu_T(\{gx\})-\nu_T(\{x\})| \ge \sum_{T \in A_g}\nu(T)\int_T|\nu_T(\{gx\})-\nu_T(\{x\})| \ge \sum_{T \in A_g}\nu(T)\sqrt{\epsilon}|Q|\text{, so}\\
        &\sum_{T \in A_g}\nu(T) < 4\sqrt{\epsilon}|Q|^{-1}, \sum_{T \in B_g}\nu(T) > 1-4\sqrt{\epsilon}|Q|^{-1}\text{, and }\sum_{T \in \cap_{g \in Q}B_g}\nu(T) > 1-4\sqrt{\epsilon}.
    \end{alignat*}
    Consequently, we let $D$ denote the union of all tiles $T$ that are contained in every $B_g$ with $g \in Q$. If $D$ is an infinite set, then using monotonicity of the measure $\nu$ we can pick a subset, which by abuse of notation we also call $D$, so that $D$ is a finite union of tiles of $\mathcal{T}$ (and hence it is finite) and satisfies $\nu(D) > 1-4\sqrt{\epsilon}$.
\end{proof}

\begin{lemma}\label{TheObviousLemma}
    Let $G$ be a countably infinite amenable group. For each finite set $F \subseteq G$ and each $\epsilon > 0$, there exists a finite set $K \subseteq G$ such that for any $(K,\epsilon)$-invariant probability measure $\nu$, we have $\nu(Fc) < 2\epsilon$ for all $c \in G$.
\end{lemma}

\begin{proof}
    Let $L \in \mathbb{N}$ be such that $L^{-1} < \epsilon$. Let $K := \{g_i\}_{i = 1}^L \subseteq G$ be such that $g_iF\cap g_jF = \emptyset$ when $i \neq j$. We see that for every $c \in G$ we have

    \begin{equation}
        L\nu(Fc) \le \sum_{i = 1}^L\left(\nu(g_iFc)+\epsilon\right) \le 1+L\epsilon\text{, hence }\nu(Fc) \le L^{-1}+\epsilon < 2\epsilon.
    \end{equation}
\end{proof}

\begin{lemma}\label{FinitizedConvergenceInMeasure}
    Let $G$ be an amenable group, $(X,\mathscr{B},\mu,(\tau_g)_{g \in G})$ an ergodic measure preserving system, and let $f \in L^1(X,\mu)$. Given $\epsilon > 0$ there exists a finite $K \subseteq G$ and a $\delta > 0$ such that for any $(K,\delta)$-invariant probability measure $\nu$ on $G$, there exists a set $A \in \mathscr{B}$ with $\mu(A) > 1-\epsilon$ such that for all $x \in A$ we have

    \begin{equation}\label{ConvergenceInMeasureEquation1}
        \left|\int_Gf(\tau_gx)d\nu(g)-\int_Xfd\mu\right| < \epsilon.
    \end{equation}
    Furthermore, if $f \in L^\infty(X,\mu)$, then we can choose $A$ so that for all $x \in A$ we also have

    \begin{equation}\label{ConvergenceInMeasureEquation2}
        \sup_{g \in \text{supp}(\nu)}|f(\tau_gx)| \ge ||f||_\infty-\epsilon.
    \end{equation}
\end{lemma}

\begin{proof}
    We begin with the case in which $f \in L^1(X,\mu)$. Let $K_1 \subseteq K_2 \subseteq \cdots \subseteq G$ be an exhaustion of $G$ by finite sets, and let $\delta_1 > \delta_2 > \cdots > \delta_n > \cdots$ tend to $0$. Let us assume for the sake of contradiction that there exists some $\epsilon > 0$ such that for each $n \in \mathbb{N}$ there exists a set $A_n \in \mathscr{B}$ with $\mu(A_n) > \epsilon$ and a $(K_n,\delta_n)$-invariant probability measure $\nu_n$ on $G$ such that

    \begin{equation}
        \left|\int_Gf(\tau_gx)d\nu_n(g)-\int_Xfd\mu\right| > \epsilon
    \end{equation}
    for all $x \in A_n$. Since $(\nu_n)_{n = 1}^\infty$ is a Reiter sequence, the Mean Ergodic Theorem (see, e.g. \cite[Proposition 5.4]{AmenabilityByPaterson}) tells us that

    \begin{equation}
        \lim_{n\rightarrow\infty}\int_Gf(\tau_gx)d\nu_n(g) = \int_Xfd\mu,
    \end{equation}
    with convergence taking place in $L^1(X,\mu)$. In particular, we have convergence in measure, so let $N \in \mathbb{N}$ be such that for all $n \ge N$ we have

    \begin{equation}
        \left|\int_Gf(\tau_gx)d\nu_n(g)-\int_Xfd\mu\right| < \epsilon
    \end{equation}
    on a set of measure at least $1-\epsilon$, which yields the desired contradiction.

    Now let us assume that $f \in L^\infty(X,\mu)$. Let $A_0 \in \mathscr{B}$ be such that $\mu(A_0) > 1-2^{-1}\epsilon$, and Equation \eqref{ConvergenceInMeasureEquation1} is satisfied for $\frac{\epsilon}{2}, f$ and all $x \in A_0$. For each $p \in \mathbb{N}$, let $A_p \in \mathscr{B}$ be such that $\mu(A_p) > 1-2^{-p-1}\epsilon$, and Equation \eqref{ConvergenceInMeasureEquation1} is satisfied for $2^{-p-1}\epsilon, |f|^p$ and all $x \in A_p$. Let $A = \bigcap_{p = 0}^\infty A_p$. We see that for any $x \in A$ and any $p \in \mathbb{N}$, there exists $g \in \text{supp}(\nu)$ for which $|f(\tau_gx)| > ||f||_p-\epsilon$. The desired result follows from the fact that $||f||_\infty = \lim_{p\rightarrow\infty}||f||_p$.
\end{proof}

\subsection{Koopman representations for positive definite functions}
Let $G$ be a locally compact second countable (l.c.s.c.) topological group with identity $e$ and left Haar measure $\lambda$. 
A function $f:G\rightarrow\mathbb{C}$ is \textbf{positive definite} if for any $c_1,\cdots,c_n \in \mathbb{C}$ and $g_1,\cdots,g_n \in G$, we have $\sum_{i,j = 1}^nc_i\overline{c_j}f(g_j^{-1}g_i) \ge 0$. 
We denote the set of all positive definite functions on $G$ by $\mathbf{P}(G)$. 
If $U$ is a unitary representation of a group $G$ on a Hilbert space $\mathcal{H}$, then $\xi \in \mathcal{H}$ is a \textbf{cyclic vector} if $\overline{\text{span}\{U_g\xi\ |\ g \in G\}} = \mathcal{H}$. 
A classical result of Gelfand, Naimark, and Segal lets us associate to each $\phi \in \mathbf{P}(G)$ a corresponding unitary representation of a l.c.s.c. group $G$.

\begin{theorem}[{\cite[Theorem C.4.10]{Kazhdan'sPropertyT}}]\label{GNSConstruction}
    If $\phi \in \mathbf{P}(G)$ then there exists a triple $(U,\mathcal{H},\xi)$ consisting of a unitary representation $U$ of $G$ on a Hilbert space $\mathcal{H}$ and a cyclic vector $\xi \in \mathcal{H}$ such that $\phi(g) = \langle U_g\xi,\xi\rangle$.
\end{theorem}

For $\phi \in \mathbf{P}(G)$, we call the triple $(U,\mathcal{H},\xi)$ given to us by Theorem \ref{GNSConstruction} the \textbf{GNS triple associated to $\phi$}. 

The Gaussian Measure Space Construction (cf. \cite[Chapter 3.11]{ErgodicTheoryViaJoinings} or \cite[Chapter 8.2]{CornfeldFominSinai}) gives us the following variation of Theorem \ref{GNSConstruction}.

\begin{theorem}\label{GaussianMPSConsequence}
    For each $\phi \in \mathbf{P}(G)$ there exists a m.p.s. $\mathcal{X} := (X,\mathscr{B},\mu,(\tau_g)_{g \in G})$ and an $f \in L^2(X,\mu)$ with the following properties:
    \begin{enumerate}[(i)]
        \item The function $f$ has a Guassian distribution, so it is unbounded.
        
        \item We have $\phi(g) = \langle \tau_gf,f\rangle$ for all $g \in G$.

        \item If $\phi$ is real-valued, then $f$ can be taken to be real-valued.

        \item If $\mathcal{X}$ is ergodic, then it is weakly mixing.

        \item If $f$ is orthogonal to all finite dimensional $(\tau_g)_{g \in G}$-invariant subspaces of $L^2(X,\mu)$, then $\mathcal{X}$ is weakly mixing.
    \end{enumerate} 
\end{theorem}

We see that if $G = \mathbb{Z}$ and $\phi \in \mathbf{P}(\mathbb{Z})$ is given by $\phi(n) = e^{2\pi in\sqrt{2}}$, then the Gaussian Measure Space Construction gives us a m.p.s. $\mathcal{X} := (X,\mathscr{B},\mu,(\tau^n)_{n \in \mathbb{Z}})$ and a $f \in L^2(X,\mu)$ for which $\langle \tau^nf,f\rangle = e^{2\pi in\sqrt{2}}$. Since $f$ is an eigenvector of $\tau$ for the eigenvalue $e^{2\pi i\sqrt{2}}$, we see that $\mathcal{X}$ is not weakly mixing, so it will not be ergodic either. 
Consequently, it is natural to ask whether or not any positive definite function $\phi \in \textbf{P}(G)$ can be represented as $\phi(g) = \langle \tau_gf,f\rangle$ with $f \in L^2(X,\mu)$ and $\mathcal{X}$ ergodic. For $G = \mathbb{Z}$ this question was answered in the positive as \cite[Lemma 5.2.1]{SohailsPhDThesis}. Our next result extends this to all $G$.

\begin{theorem}\label{ErgodicRepresentation}
    Let $G$ be a l.c.s.c. group and let $\phi \in \mathbf{P}(G)$. There exists an ergodic m.p.s. $(X,\mathscr{B},\mu,(\tau_g)_{g \in G})$ and $f \in L^2(X,\mu)$ such that $\phi(g) = \langle \tau_gf,f\rangle$. Furthermore, if $\phi$ is real-valued, then $f$ can also be taken to be real-valued.
\end{theorem}

\begin{proof}[Proof of Theorem \ref{ErgodicRepresentation}]
    Let $\phi$ take values in $\mathbb{K} \in \{\mathbb{R},\mathbb{C}\}$. By Theorem \ref{GNSConstruction} let $U$ be a unitary representation of $G$ in a Hilbert space $\mathcal{H}$ and $f' \in \mathcal{H}$ a cyclic vector for which $\phi(g) = \langle U_gf',f'\rangle$. Let $\mathcal{H} = \mathcal{H}_c\oplus\mathcal{H}_w$ be the decomposition in which $\mathcal{H}_w$ has no finite dimensional $U$-invariant subspaces, and $\mathcal{H}_c$ decomposes into a direct sum of finite dimensional $U$-invariant subspaces. Let $f' = f_c'+f_w'$ with $f_c' \in \mathcal{H}_c$ and $f_w' \in \mathcal{H}_w$.
        
    We would now like to verify that $\langle U_gf'_c,f'_c\rangle$ and $\langle U_gf'_w,f'_w\rangle$ take values in $\mathbb{K}$. Since this is clear if $\mathbb{K} = \mathbb{C}$, let us assume for the moment that $\mathbb{K} = \mathbb{R}$. 
    Let us further assume for the sake of contradiction that $\left|\text{Im}(\langle U_{g_0}f'_c,f'_c\rangle)\right| = \epsilon > 0$ for some $g_0 \in G$ and $\epsilon > 0$. Since $g\mapsto \langle U_gf'_c,f'_c\rangle$ is an almost periodic function, we see that 

    \begin{equation}
        S := \left\{g \in G\ |\ \left|\text{Im}(\langle U_gf'_c,f'_c\rangle)\right| > \frac{\epsilon}{2}\right\},
    \end{equation}
    is syndetic. 
    Let $g_1,\cdots,g_r \in G$ be such that $G = \bigcup_{i = 1}^rg_iS$.
    Since $f'_w \in \mathcal{H}_w$, we see that for any F\o lner sequence $(F_n)_{n = 1}^\infty$ we have
    
    \begin{equation}
        \lim_{n\rightarrow\infty}\frac{1}{|F_n|}\sum_{g \in S\cap F_n}|\text{Im}(\langle U_gf'_w,f'_w\rangle)| \le \lim_{n\rightarrow\infty}\frac{1}{|F_n|}\sum_{g \in F_n}|\langle U_gf'_w,f'_w\rangle| = 0.
    \end{equation}
    Since $\phi$ is real-valued, we must have that $\left|\text{Im}(\langle U_gf'_w,f'_w\rangle)\right| = \left|-\text{Im}(\langle U_gf'_c,f'_c\rangle)\right| > \frac{\epsilon}{2}$ for all $g \in S$.
    However, this implies that

    \begin{equation}
        \lim_{n\rightarrow\infty}\frac{1}{|F_n|}\sum_{g \in S\cap F_n}|\text{Im}(\langle U_gf'_w,f'_w\rangle)| > \frac{\epsilon}{2}\lim_{n\rightarrow\infty}\frac{1}{|F_n|}\sum_{g \in F_n}\mathbbm{1}_S(g) \ge \frac{\epsilon}{2r},
    \end{equation}
    which yields the desired contradiction.

    Using Theorem \ref{GaussianMPSConsequence} we may pick a weakly mixing m.p.s. $\mathcal{X}_w := (X_w,\mathscr{B}_w,\mu_w,(\tau_{w,g})_{g \in G})$ and $f_w'' \in L^2_{\mathbb{K}}(X_w,\mu_w)$ for which $\langle \tau_{w,g}f_w'',f_w''\rangle_{L^2} = \langle U_gf_w',f_w'\rangle$. To handle $f_c'$, we require the following result.

\begin{lemma}\label{MainResultForCompactGroups}
Let $\phi \in \mathbf{P}(G)$ take values in $\mathbb{K}$ and let $(U, \mathcal{H}, \xi)$ be the associated GNS-triple. 
Suppose that $\mathcal{H}$ decomposes as a direct sum of finite dimensional sub-representations. 
Then there exists an ergodic m.p.s. $(K,\mathscr{B},\lambda_K,(\tau_g)_{g \in G})$ and $F \in L^2_{\mathbb{K}}(K,\lambda_K)$ for which $\phi(g) = \langle \tau_gF,F\rangle$.
\end{lemma}

\begin{proof}[Proof of Lemma \ref{MainResultForCompactGroups}]
    Let $\mathcal{U}(\mathcal{H})$ denote the group of unitary operators on $\mathcal{H}$ with the strong operator topology. 
    Let $\mathcal{H} = \oplus_{i \in I}\mathcal{H}_i$ be a decomposition of $\mathcal{H}$ into finite dimensional irreducible subrepresentations.
    Then the unitaries $U_g$, for $g\in G$, are all contained in the natural copy of the compact group $\prod_{i \in I}\mathcal{U}(\mathcal{H}_i)$ that lives in $\mathcal{U}(\mathcal{H})$. 
    Therefore, $K := \overline{\{ U_g\}_{g\in G}}$ is a compact subgroup of $\mathcal{U}(\mathcal{H})$, and $\phi$ factors through the homomorphism from $G$ to $K$ and extends there to the continuous positive definite function $\phi '$ on $K$ via $\phi ' (k)=\langle k\xi, \xi \rangle$. 
    Letting $\lambda_K$ denote the normalized Haar measure of $K$, by \cite[Lemma 14.1.1]{C(star)AlgebraByDixmier} there exists $F \in L^2_{\mathbb{K}}(K,\lambda_K)$ for which $\phi'(k) = \langle L_kF,F\rangle$, where $L$ is the left regular representation of $K$. 
    Letting $\tau_g = L_{U_g}$ we see that $\langle \tau_gF,F\rangle = \phi'(U_g) = \langle U_g\xi,\xi\rangle = \phi(g)$, so it only remains to observe that $(K,\mathscr{B},\lambda_K,(\tau_g)_{g \in G})$ is ergodic, since the image of $G$ in $K$ is dense.
\end{proof}
    
    Using Lemma \ref{MainResultForCompactGroups} we may pick an ergodic m.p.s. $\mathcal{X}_c := (X_c,\mathscr{B}_c,\mu_c,(\tau)_{c,g})_{g \in G})$ and $f_c'' \in L^2_{\mathbb{K}}(X_c,\mu_c)$ for which $\langle \tau_{c,g}f_c'',f_c''\rangle_{L^2} = \langle U_gf_c',f_c'\rangle$. Now let $\mathcal{X} = \mathcal{X}_c\times\mathcal{X}_w$ and note that $\mathcal{X}$ is ergodic. Let $f_w,f_c \in L^2_{\mathbb{K}}(X,\mu)$ be given by $f_w(x_1,x_2) = f_w''(x_1)$ and $f_c(x_1,x_2) = f_c''(x_2)$, and observe that $\int_Xf_wd\mu_w\times\mu_c = \int_{X_w}f''_wd\mu_w = 0$. We see that for $f = f_w+f_c$ we have

    \begin{alignat*}{2}
        \langle \tau_gf,f\rangle &= \langle \tau_{w,g}f_w,f_w\rangle+\langle \tau_{w,g}f_w,f_c\rangle+\langle \tau_{c,g}f_c,f_w\rangle+\langle \tau_{c,g}f_c,f_c\rangle\\
        &= \langle U_gf_w',f_w'\rangle+\int_{X_w}\tau_{w,g}f''_wd\mu_w\int_{X_c}f''_cd\mu_c+\int_{X_w}f''_wd\mu_w\int_{X_c}\tau_{c,g}f''_cd\mu_c+\langle U_gf_c',f_c'\rangle\\
        &= \langle U_gf_w',f_w'\rangle+\langle U_gf_c',f_c'\rangle=\langle U_gf,f\rangle = \phi(g).
    \end{alignat*}
\end{proof}

\begin{remark}
    It is natural to ask if we can improve Theorem \ref{ErgodicRepresentation} by requiring that $f \in L^\infty$ instead of $f \in L^2$. 
    It is a classical result of Foia\c{s} and Str\u{a}til\u{a} \cite{KroneckerImpliesGaussian} (see also \cite[Theorem 14.4.2$'$]{CornfeldFominSinai}) that if $E \subseteq [0,1]$ is a Kronecker set, $\nu$ a continuous measure supported on $E\cup(1-E)$, and $(X,\mathscr{B},\mu,(\tau^n)_{n \in \mathbb{Z}})$ is an ergodic m.p.s. with some $f \in L^2(X,\mu)$ for which $\hat{\nu}(n) = \langle \tau^nf,f\rangle$, then $f$ has a Gaussian distribution. 
    It follows that the function $f$ given to us by Theorem \ref{ErgodicRepresentation} applied to such a measure $\nu$, will not be in $L^\infty$.
\end{remark}

\subsection{The unique invariant mean on the space of weakly almost periodic functions}\label{SectionOnMeans}
A general treatment of (weak) almost periodicity for vector-valued functions is given in \cite{APFunctionsAndRepresentations}.
Here we collect some facts that we will use about weakly almost periodic functions taking values in a Hilbert space.
For simplicity, we restrict our attention to countably infinite groups $G$, as this is also the level of generality that suffices for our applications in Section \ref{Section:Appendix}.
Let $\mathcal{H}$ be a Hilbert space and let $\ell^\infty(G,\mathcal{H})$ denote the set of bounded functions $f:G\rightarrow\mathcal{H}$.
Let $L$ denote the left regular representation of $G$ on $\ell^\infty(G,\mathcal{H})$, i.e., $(L_hf)(g) = f(h^{-1}g)$.
A function $f \in \ell^\infty(G)$ is \textbf{weakly almost periodic} if the set $\{L_gf\}_{g \in G}$ is relatively weakly compact.\footnote{Some sources define this through the use of the left regular antirepresentation $L'$ that is given by $(L'_hf)(g) = f(hg)$, but both definitions are equivalent.} 
We let $W(G,\mathcal{H})$ denote the collection of weakly almost periodic functions in $\ell^\infty(G,\mathcal{H})$.
It is well know that there is a unique left invariant mean $M$ on $W(G) = W(G,\mathbb{C})$ (see for example \cite[Chapter 3.1]{GreenleafBook}), i.e., a positive linear functional of norm $1$ satisfying $M(L_hf) = M(f)$ for all $h \in G$ and all $f \in W(G)$.
Furthermore, the mean $M$ will also be right invariant, hence we simply refer to $M$ as the unique invariant mean.

From the invariant mean $M:W(G)\rightarrow\mathbb{C}$ we construct an invariant operator $M':W(G,\mathcal{H})\rightarrow\mathcal{H}$ as follows.
Write $\mathcal{H}^* = \{\eta^*\ |\ \eta \in \mathcal{H}\}$ for the dual space of $\mathcal{H}$, i.e. $\eta^*(\xi) = \langle \xi,\eta\rangle$ for all $\xi \in \mathcal{H}$ and all $\eta^* \in \mathcal{H}^*$.
We observe that $\mathcal{H}^*$ and $\mathcal{H}^{**} := (\mathcal{H}^*)^*$ are both naturally isomorphic to $\mathcal{H}$ as Hilbert spaces.
Observe that for any $f \in W(G,\mathcal{H})$ and any $\eta^* \in \mathcal{H}^*$, the function $f_{\eta^*}(g) = \langle f(g),\eta\rangle$ is in $W(G)$, so we may define a map $M':W(G,\mathcal{H})\rightarrow\mathcal{H}^{**}$ by $M(f_{\eta^*}) = \langle M'(f),\eta\rangle$ for all $f \in W(G,\mathcal{H})$ and all $\eta^* \in \mathcal{H}^*$.
To see that $M'$ is left invariant, we see that for all $f \in W(G,\mathcal{H})$, all $\eta^* \in \mathcal{H}^*$, and all $h \in G$, we have 

\begin{equation}
    \langle M'(L_hf),\eta\rangle = M(L_hf_{\eta^*}) = M(f_{\eta^*}) = \langle M'(f),\eta\rangle.
\end{equation}
A similar calculation shows that $M'$ is also right invariant. 
By abuse of notation, we also write $M':W(G,\mathcal{H})\rightarrow\mathcal{H}$ after identifying $\mathcal{H}^{**}$ with $\mathcal{H}$.
It can also be checked that $M'(f)$ belongs to the closed convex hull of the range of $f$.

If $U$ is a unitary representation of $G$ on $\mathcal{H}$, then $U(G)$ is a relatively weakly compact subset of the space of all bounded linear operators on $\mathcal{H}$.
Consequently, for any $\xi \in \mathcal{H}$, we have that $f_1(g) = U_g\xi$ is a weakly almost periodic function in $\ell^\infty(G,\mathcal{H})$, and that $f_2(g) = \langle U_g\xi,\xi\rangle$ is a weakly almost periodic function in $\ell^\infty(G) = \ell^\infty(G,\mathbb{C})$.
Let $\mathcal{H} = \mathcal{H}_I\oplus\mathcal{H}_I^{\perp}$, where $\mathcal{H}_I$ is the subspace of $U$-invariant vectors.
We see that for $\xi_I \in \mathcal{H}_I$, we have 

\begin{equation}
    M(g\mapsto\langle U_g\xi_I,\xi_I\rangle) = M\left(g\mapsto||\xi_I||^2\right) = ||\xi_I||^2\text{ hence }M'(g\mapsto U_g\xi_I) = \xi_I.
\end{equation}
Similarly, if $\xi_e \in \mathcal{H}_I^{\perp}$, then $M'(g\mapsto U_g\xi_e) = 0$, because the closed convex hull of $\{U_g\xi_e\}_{g \in G}$ is contained in $\mathcal{H}_I^{\perp}$, and the only $U$-invariant vector in $\mathcal{H}_I^{\perp}$ is $0$.
It follows that for any $\xi \in \mathcal{H}$, we have $M'(g\mapsto U_g\xi) = P_I\xi$ and $M(g\mapsto\langle U_g\xi,\xi\rangle) = ||P_I\xi||^2$, where $P_I:\mathcal{H}\rightarrow\mathcal{H}_I$ is the orthogonal projection.

It is also worth observing that the Jacobs-de Leeuw-Glicksberg decomposition can be expressed in terms of the mean $M$. 
In particular, we have $\mathcal{H} = \mathcal{H}_c\oplus\mathcal{H}_w$ where $\mathcal{H}_c$ is the direct sum of all finite dimensional subrepresentations of $U$, and $\mathcal{H}_w = \{\xi \in \mathcal{H}\ |\ M(g\mapsto|\langle U_g\xi,\xi\rangle|) = 0\}$. 
We refer the reader to \cite[Appendix 10]{ATempelmanBook} for more details regarding this particular formulation of the Jacobs-de Leeuw-Glicksberg decomposition and how it connects to the previous work of Godement.

\section{Asymptotic dynamics on amenable groups}

We begin by recalling a result that appeared implicitly in the work of Ruzsa \cite{RuzsavanderCorput}.

\begin{theorem}\label{Ruzsa'sResult}
    Let $\phi:\mathbb{Z}\rightarrow\mathbb{C}$ be a positive definite sequence satisfying $\phi(0) = 1$. There exists $(c_n)_{n = 1}^\infty \subseteq \mathbb{S}^1$ for which

    \begin{equation}
        \phi(h) = \lim_{N\rightarrow\infty}\frac{1}{N}\sum_{n = 1}^N c_{n+h}\overline{c_n}.
    \end{equation}
\end{theorem}

We want to generalize Ruzsa's result to any countably infinite amenable group $G$ and any Reiter sequence $(\nu_n)_{n = 1}^\infty$ in $G$. To this end, we begin by reviewing the ideas behind the proof of Theorem \ref{Ruzsa'sResult}, as they will also be present in our generalization. We remark that Ruzsa used the language of probability to prove his result, and the following discussion uses the language of ergodic theory.

Firstly, we observe that there exists a probability measure $\mu$ on $\mathbb{T}$ for which $\phi(h) = \hat{\mu}(h)$. We then see that for the Hilbert space $\mathcal{H} = L^2(\mathbb{T},\mu)$, there is a natural unitary operator $U:\mathcal{H}\rightarrow\mathcal{H}$ given by $U(f)(x) = e^{2\pi ix}f(x)$, and that $\hat{\mu}(h) = \langle U^h1,1\rangle$. The operator $U$ is a multiplication operator, and we want to convert it into a Koopman operator so that we can use the Birkhoff Pointwise Ergodic Theorem to model the global dynamics of a given function through the pointwise orbits of that function. 
Consequently, we now consider $\mathcal{H}' = L^2(\mathbb{T}\times\mathbb{T},\mu\times m)$, where $m$ is the Lebesgue measure.
We see that $\tau(x,y) = (x,y+x)$ is a measure preserving automorphism of $\mathbb{T}\times\mathbb{T}$, and that for $f:\mathbb{T}\times\mathbb{T}\rightarrow\mathbb{S}^1$ given by $f(x,y) = e^{2\pi iy}$, we have $\langle\tau^hf,f\rangle = \int_{\mathbb{T}}e^{2\pi i hx}d\mu(x) = \hat{\mu}(h)$. 
If the transformation $\tau$ was ergodic, then we could take $c_n = f(\tau^nx)$ for some generic point $x$, but it is unfortunately clear that the transformation $\tau$ is in general highly non-ergodic. 
However, the ergodic decomposition of $\tau$ is easy to see from the given presentation.

Now suppose that we want to approximate the values of $\phi(h)$ up to a precision of $\epsilon$ for all $h \in H$ with $H$ finite, and some fixed $N = N_0$, $(c_n)_{n = 1}^{N_0+\text{max}(H)}$. We take $N_0$ to be so large that it can be partitioned into a large number of intervals of size $M$, with $M$ also sufficiently large. We approximate $f$ by a simple function in which the dynamics of each of the constituent step functions can be modeled by the restriction of that step function to some ergodic component. Since $M$ is sufficiently large, the dynamics of the restricted step function can be modeled by some sequence $(c_n)_{n = 1}^M \subseteq \mathbb{S}^1$ as a consequence of Birkhoff's Theorem. We then associate each of the $\frac{N_0}{M}$ intervals of length $M$ to one of the step functions, and the frequency with which we do so is dictated by $\mu$, because $\mu$ tells us how much weight to give each ergodic component. We then stitch together a sequence of finitistic approximations to get the desired result globally.

\begin{lemma}\label{KeyLemmaForAmenableGroups}
   Let $G$ be a countably infinite amenable group, let $H \subseteq G$ be finite with $e \in H$, let $\epsilon > 0$ be arbitrary, and let $(X,\mathscr{B},\mu,(\tau_g)_{g \in G})$ be a measure preserving system. Fix $f \in L^2(X,\mu)$ and let $R \subseteq \text{Range}(f)$ be a dense subset. There exists a $\delta > 0$, a finite set $K \subseteq G$, and a sequence $(c_g)_{g \in G} \subseteq R$ with $||(c_g)_{g \in G}||_\infty$ bounded by a function of $f$ and $\epsilon$, such that for every $(K,\delta)$-invariant probability measure $\nu$ we have

   \begin{alignat}{2}
       &\int_G|c_g|^2d\nu(g) \overset{\epsilon}{=} ||f||_2^2,\ \int_Gc_gd\nu(g) \overset{\epsilon}{=} \int_Xfd\mu,\text{ and }\label{KeyLemmaEquation1}\\
       &\int_Gc_{hg}\overline{c_g}d\nu(g) \overset{\epsilon}{=} \langle \tau_hf,f\rangle\text{ for all }h \in H.\label{KeyLemmaEquation2}
   \end{alignat}
   Furthermore, if $f \in L^\infty(X,\mu)$, then for any $h_1,\cdots,h_\ell \in H$ and $t_1,\cdots,t_\ell \in [0,|H|]$ we have

   \begin{alignat}{2}
       &\int_Gd_{h_1g,1}^{t_1}\cdots d_{h_\ell g,\ell}^{t_\ell}d\nu(g) \overset{\epsilon}{=} \int_X \tau_{h_1}f_1^{t_1}\cdots \tau_{h_\ell}f_\ell^{t_\ell}d\mu\text{, and}\label{KeyLemmaEquation3}\\
       &||(d_{h_1g,1}^{t_1}\cdots d_{h_\ell g,\ell}^{t_\ell})_{g \in G}||_\infty \overset{\epsilon}{=} ||\tau_{h_1}f_1^{t_1}\cdots \tau_{h_\ell}f_\ell^{t_\ell}||_\infty,\label{KeyLemmaEquation4}
   \end{alignat}
   where $f_i \in \{f,\overline{f}\}$ and $(d_{g,i})_{g \in G} \in \left\{(c_g)_{g \in G},(\overline{c}_g)_{g \in G}\right\}$, and $f_i = f$ if and only if $(d_{g,i})_{g \in G} = (c_g)_{g \in G}$.
\end{lemma}

\begin{proof}
    We give the proof for Equation \ref{KeyLemmaEquation2} as well as Equation \eqref{KeyLemmaEquation4} in the corresponding case, and remark that the proof for Equation \eqref{KeyLemmaEquation3} is similar. 
    Let $f' \in L^\infty(X,\mu)$ be such that Range$(f') \subseteq R$, $||f'-f||_2 < \frac{\epsilon}{16||f||_2}$ and $||f'||_\infty = M$. 
    We begin by taking the ergodic decomposition of $(X,\mathscr{B},\mu,(\tau_g)_{g \in G})$. 
    Let $\mathcal{Y} := (Y,\mathscr{A},\gamma)$ be a probability space such that $(X,\mathscr{B},\mu,(\tau_g)_{g \in G})$ is the direct integral over $\mathcal{Y}$ of the ergodic systems $\mathcal{X}_y := (X_y,\mathscr{B}_y,\mu_y,(\tau_{y,g})_{g \in G})$. 
    Since $\tau_{y,g} = \tau_g|_{X_y}$, we will simply write $\tau_g$ instead of $\tau_{y,g}$ to save on notation. 
    Let $f_y \in L^\infty(X_y,\mu_y)$ be given by $f_y = f'|_{X_y}$. For $h \in H$, let $f_h:Y\rightarrow\mathbb{C}$ be given by $f_h(y) = \int_{X_y}\tau_hf_y(x)\overline{f_y(x)}d\mu_y(x)$, and let $S_h = \sum_{j = 1}^{J_h}w_{j,h}\mathbbm{1}_{Y_{j,h}}$ be a simple function on $Y$ with $\{Y_{j,h}\}_{j = 1}^{J_h}$ being pairwise disjoint and $||S_h-f_h||_\infty < \frac{\epsilon}{8}$.
    Let $J(H) = \{(j_h)_{h \in H}\ |\ 1 \le j_h \le J_h\ \forall\ h \in H\}$, and for each $\vec{j} \in J(H)$ let $Y_{\vec{j}} := \bigcap_{h \in H}Y_{j_h,h}$, and if $Y_{\vec{j}} \neq \emptyset$ let $y_{\vec{j}} \in Y_{\vec{j}}$ be such that
    
    \begin{equation}\label{EnsuringInftyNormEquation}
        ||\tau_hf_{y_{\vec{j}}}\overline{f_{y_{\vec{j}}}}||_\infty > \sup_{y \in Y_{\vec{j}}}||\tau_hf_y\overline{f_y}||_\infty-\frac{\epsilon}{2}.
    \end{equation}
    Let $K_{\vec{j},h}$, $\delta_{\vec{j},h}$, and $A_{\vec{j},h}$ be as in Lemma \ref{FinitizedConvergenceInMeasure} with respect to $\frac{\epsilon}{8|H|}$ and $\tau_hf_{y_{\vec{j}}}\overline{f_{y_{\vec{j}}}}$. 
    Let $K_1 = \bigcup_{h \in H}\bigcup_{\vec{j} \in J(H)}K_{\vec{j},h}$ and for each $\vec{j} \in J(H)$ let $x_{\vec{j}} \in \bigcap_{h \in H}A_{j_h,h}$ be arbitrary. 
    We require that $\sqrt{\delta}|K_1| < \text{min}\left\{\delta_{\vec{j},h}\ |\ h \in H\ \&\ \vec{j} \in J(H)\right\}$, $8M^2\sqrt{\delta} < \frac{\epsilon}{8}$, and $\delta < \frac{\epsilon}{16|J(H)|}$.

    Let $\mathcal{T}$ be a tiling of $G$ whose shapes $\{T_i\}_{i = 1}^I$ are each $(K_1H^{-1},\delta)$-invariant, and let $U = \bigcup_{i = 1}^IT_i$.
    Let $K_2 \subseteq G$ be as in Lemma \ref{TheObviousLemma}, with respect to $U$ and $\frac{\epsilon}{16|J(H)|}$, and let $K = HTT^{-1}\cup K_2$. 
    Let $C = \bigcup_{i = 1}^IC(T_i)$, where $C(T_i)$, $i \in \{1,\cdots,I\}$ are the center sets of the tiling. Now consider a partition $C = \bigsqcup_{\vec{j} \in J(H)}D_{\vec{j}}$ for which we have

    \begin{equation}\label{EquationForPartition}
         \left|\gamma\left(Y_{\vec{j}}\right)-\sum_{i = 1}^I\sum_{a \in D_{\vec{j}}\cap C(T_i)}\nu(T_ia)\right| < \frac{\epsilon}{8|J(H)|}\text{ for all }\vec{j} \in J(H).
    \end{equation} 
    Furthermore, we may assume without loss of generality that $D_{\vec{j}} = \emptyset$ if $\gamma(Y_{\vec{j}}) = 0$. To see that the choice of $D_{\vec{j}}$ can be made independently of the $(K,\delta)$-invariant measure $\nu$, we observe that $D_{\vec{j}}$ can be chosen by only making use of the fact that $\nu(Uc) < \frac{\epsilon}{8|J(H)|}$ for all $c \in G$. For each $\vec{j} \in J(H)$, let $D_{\vec{j}} = \bigcap_{h \in H}D_{j_h,h}$. 
    For $a \in C(T_i)\cap D_{\vec{j}}$ and $g \in T_i$, let $c_{ga} = f_{y_{\vec{j}}}(\tau_{ga}x_{\vec{j}})$. 
    Using Lemma \ref{GoodTilingLemma}, let $D \subseteq G$ be a finite union of tiles of $\mathcal{T}$ for which $\nu(D) > 1-4\sqrt{\delta}$ and for every tile $T \subseteq D$ the probability measure $\nu_T$ is $(K_1,\sqrt{\delta}|K_1|)$-invariant. 
    Let $C_i = C(T_i)\cap D$. Let us now verify that Equation \eqref{KeyLemmaEquation2} holds. Fix $h \in H$ and observe that

    \begin{alignat*}{2}
        &\int_Gc_{hg}\overline{c_g}d\nu(g) \overset{4M^2\sqrt{\delta}}{=} \int_Dc_{hg}c_gd\nu(g) = \sum_{i = 1}^I\sum_{a \in C_i}\int_{T_i}c_{hga}\overline{c_{ga}}d\nu(ga)\\
        \overset{4M^2\delta}{=} & \sum_{i = 1}^I\sum_{a \in C_i}\int_{T_i\cap h^{-1}T_i}c_{hga}\overline{c_{ga}}d\nu(ga) = \sum_{\vec{j} \in J(H)}\sum_{i = 1}^I\sum_{a \in C_i\cap D_{\vec{j}}}\int_{T_ia\cap h^{-1}T_ia}f_{y_{\vec{j}}}(\tau_{hg}x_{\vec{j}})\overline{f_{y_{\vec{j}}}(\tau_{g}x_{\vec{j}})}d\nu(g)\\
        \overset{4M^2\delta}{=} & \sum_{\vec{j} \in J(H)}\sum_{i = 1}^I\sum_{a \in C_i\cap D_{\vec{j}}}\int_{T_ia}f_{y_{\vec{j}}}(\tau_{hg}x_{\vec{j}})\overline{f_{y_{\vec{j}}}(\tau_{g}x_{\vec{j}})}d\nu(g)\\
        = &\sum_{\vec{j} \in J(H)}\sum_{i = 1}^I\sum_{a \in C_i\cap D_{\vec{j}}}\nu(T_ia)\int_{T_ia}f_{y_{\vec{j}}}(\tau_{hg}x_{\vec{j}})\overline{f_{y_{\vec{j}}}(\tau_{g}x_{\vec{j}})}d\nu_{T_ia}(g)\\
        \overset{\frac{\epsilon}{8}}{=} &\sum_{\vec{j} \in J(H)}\sum_{i = 1}^I\sum_{a \in C_i\cap D_{\vec{j}}}\nu(T_ia)\int_{X_{y_{\vec{j}}}}\tau_hf_{y_{\vec{j}}}\overline{f_{y_{\vec{j}}}}d\mu_{y_{\vec{j}}} \overset{4M^2\sqrt{\delta}}{=} \sum_{\vec{j} \in J(H)}\sum_{i = 1}^I\sum_{a \in C(T_i)\cap D_{\vec{j}}}\nu(T_ia)f_h(y_{\vec{j}})\\
        \overset{\frac{\epsilon}{8}}{=} & \sum_{\vec{j} \in J(H)}\sum_{i = 1}^I\sum_{a \in C(T_i)\cap D_{\vec{j}}}\nu(T_ia)S_h(y_{\vec{j}})\overset{\frac{\epsilon}{8}}{=} \sum_{\vec{j} \in J(H)}\gamma(Y_{\vec{j}})S_h(y_{\vec{j}}) = \sum_{j = 1}^{J_h}\gamma(Y_{j,h})w_{j,h}\\
        =& \int_YS_hd\gamma \overset{\frac{\epsilon}{8}}{=} \int_Yf_hd\gamma \overset{\frac{\epsilon}{8}}{=} \langle \tau_hf,f\rangle.
    \end{alignat*}
    Lastly, we will verify that

    \begin{equation}
        ||(c_{hg})_{g \in G}||_\infty \ge ||\tau_hf\overline{f}||_\infty-\frac{\epsilon}{2}.
    \end{equation}
    Pick $\vec{j} \in J(H)$ such that $||\tau_hf_{y_{\vec{j}}}\overline{f_{y_{\vec{j}}}}||_\infty \ge ||\tau_hf\overline{f}||_\infty-\epsilon$. Since any tile $T$ of $\mathcal{T}$ is $(K_{\vec{j},h}H^{-1},\frac{1}{2}\delta_{\vec{j},h})$-invariant, we see that $T\cap hT$ is $(K_{\vec{j},h},\delta_{\vec{j},h})$-invariant. Since $x_{\vec{j}} \in A_{\vec{j},h}$, we see that for $a \in C(T_i)\cap D_{\vec{j}}$ we have

    \begin{equation}
         \sup_{g \in T_ia\cap h^{-1}T_ia}|c_{hg}\overline{c_g}| = \sup_{g \in T_ia\cap h^{-1}T_ia}|f_{y_{\vec{j}}}(\tau_{hg}x_{\vec{j}})\overline{f_{y_{\vec{j}}}(\tau_gx_{\vec{j}})}| > ||\tau_hf_{y_{\vec{j}}}\overline{f_{y_{\vec{j}}}}||_\infty-\frac{\epsilon}{2}.
    \end{equation}
\end{proof}

\begin{theorem}\label{RepresentationTheoremForAmenableGroups}
   Let $G$ be a countably infinite amenable group, let $(\nu_n)_{n = 1}^\infty$ be a Reiter sequence, and let $(X,\mathscr{B},\mu,(\tau_g)_{g \in G})$ be a measure preserving system. Given $f \in L^2(X,\mu)$ and a dense set $R \subseteq \text{Range}(f)$, there exists a sequence of complex numbers $(c_g)_{g \in G}$ taking values in $R$ satisfying

   \begin{alignat}{2}
       &\lim_{n\rightarrow\infty}\int_G|c_g|^2d\nu_n = ||f||_2^2,\ \lim_{n\rightarrow\infty}\int_Gc_gd\nu_n = \int_Xfd\mu,\text{ and }\label{FvdCIsvdCEquation1}\\
       &\lim_{n\rightarrow\infty}\int_Gc_{hg}\overline{c_g}d\nu_n = \langle \tau_hf,f\rangle\text{ for all }h \in G.\label{FvdCIsvdCEquation2}
   \end{alignat}
   Furthermore, if $f \in L^\infty(X,\mu)$, then for any $h_1,\cdots,h_\ell \in G$ and $t_1,\cdots,t_\ell \in \mathbb{N}$ we have

   \begin{alignat}{2}
       &\lim_{n\rightarrow\infty}\int_Gd_{h_1g,1}^{t_1}\cdots d_{h_\ell g,\ell}^{t_\ell}d\nu_n = \int_X \tau_{h_1}f_1^{t_1}\cdots \tau_{h_\ell}f_\ell^{t_\ell}d\mu\text{, and}\label{FvdCIsvdCEquation3}\\
       &||(d_{h_1g,1}^{t_1}\cdots d_{h_\ell g,\ell}^{t_\ell})_{g \in G}||_\infty = ||\tau_{h_1}f_1^{t_1}\cdots \tau_{h_\ell}f_\ell^{t_\ell}||_\infty,\label{FvdCIsvdCEquation4}
   \end{alignat}   
   where $f_i \in \{f,\overline{f}\}$ and $(d_{g,i})_{g \in G} \in \left\{(c_g)_{g \in G},(\overline{c}_g)_{g \in G}\right\}$, and $f_i = f$ if and only if $(d_{g,i})_{g \in G} = (c_g)_{g \in G}$.
\end{theorem}

\begin{proof}
We give the proof of Equation \eqref{FvdCIsvdCEquation2} and remark that the proof of Equations \eqref{FvdCIsvdCEquation3} and \eqref{FvdCIsvdCEquation4} is similar. 
Let us fix an exhaustion $\{e\} \subseteq H_1 \subseteq H_2 \subseteq \cdots$ of $G$ by finite sets. 
Let $(\epsilon_q)_{q = 1}^\infty$ be a sequence decreasing to $0$, and let $(c_{g,q})_{g \in G}$ satisfy the conclusion of Lemma \ref{KeyLemmaForAmenableGroups} with respect to $f,\epsilon_q,$ and $H_q$. 
Furthermore, by allowing $\epsilon_q$ to tend to $0$ slowly enough, we assume without loss of generality that $||(c_{g,q})_{g \in G}||_\infty < 2^q$ for all $q \in \mathbb{N}$. 

Now we will construct the sequence $(c_g)_{g \in G}$ by an inductive process. 
To do this, we will also have to inductively construct a congruent sequence of tilings $(\mathcal{T}_q)_{q = 1}^\infty$, a sequence of positive real numbers $(\delta_n)_{n = 1}^\infty$ tending to $0$, an increasing sequence $(N_q)_{q = 1}^\infty \subseteq \mathbb{N}$, and increasing sequences $(V_q)_{q = 1}^\infty,(W_q)_{q = 1}^\infty,$ and $(K_n)_{n = 1}^\infty$ of finite subsets of $G$. 
Let $(\mathcal{T}_n')_{n = 1}^\infty$ be given by Theorem \ref{CongruentTilingTheorem} with respect to $(\epsilon_n)_{n = 1}^\infty$ and $(\mathcal{H}_n)_{n = 1}^\infty$. 
For the base case of this inductive procedure, let $N_1,N_2 \in \mathbb{N}$ and $\{e\} \subseteq V_1 \subseteq W_1 \subseteq V_2 \subseteq G$ and $\delta_1 \ge \delta_2 > 0$ all be arbitrary, then let $\mathcal{T}_1 = \mathcal{T}_1'$ and $\mathcal{T}_2 = \mathcal{T}_2'$. 
For $1 \le n \le N_2$, let $K_n$ be arbitrary. 
For $g \in W_1$, let $c_g \in R$ be arbitrary.
For the inductive step with $q \ge 2$, we will construct $N_{q+1}$, $V_{q+1}, W_q$, $\mathcal{T}_{q+1}, \delta_{q+1}$, define $K_n$ for $N_q < n \le N_{q+1}$, and define $c_g$ for $g \in W_q\setminus W_{q-1}$. 

Let $K_{q+1}$ and $\delta_{q+1}$ be as in Lemma \ref{KeyLemmaForAmenableGroups} with respect to $f$, $\epsilon_{q+1}$, and $H_{q+1}$. Let $\mathcal{T}_{q+1} = \mathcal{T}_k'$ for a value of $k$ so large that each tile is $(K_{q+1},\delta_{q+1}^2)$-invariant. 
Let the shapes of $\mathcal{T}_{q+1}$ be $\{\tau_{q+1,i}\}_{i = 1}^{I_{q+1}}$ and let $U_{q+1} = \bigcup_{i = 1}^{I_{q+1}}\tau_{q+1,i}$. 
Furthermore, we may assume without loss of generality that $K_{q+1} \supseteq K_qU_qU_q^{-1}$ and $\delta_{q+1} < 2^{-8q}\delta_q^2I_q^{-1}|U_q|^{-1}|K_q|^{-2}$. 
Let $W_q$ denote the union of all tiles of $\mathcal{T}_{q+1}$ that intersect $V_q$.
Using Lemma \ref{TheObviousLemma} let $N_{q+1}$ be such that for $N_{q+1} < n$ we have $\nu_n(W_q) < \delta_q2^{-4q}$ and that $\nu_n$ is $(K_{q+1}U_{q+1}U_{q+1}^{-1},2^{-8q}\delta_{q+1}^2I_{q+1}^{-1}|U_{q+1}|^{-1})$-invariant. We recall that for $n \in \mathbb{N}$ and a finite set $F \subseteq G$ for which $\nu_n(F) \neq 0$, we define $\nu_{n,F}(A) = \frac{\nu_n(A\cap F)}{\nu_n(F)}$. 
For $n \le N_{q+1}$, let $D_{n,q+1}$ be a union of tiles of $\mathcal{T}_{q+1}$ for which $\nu_n(D_{n,q+1}) > 1-4\cdot2^{-4q-4}\delta_{q+1}$, and using Lemma \ref{GoodTilingLemma} we may assume for $N_q < n \le N_{q+1}$ that for each tile $T \subseteq D_{n,q+1}$, $\nu_{n,T}$ is $(K_q,2^{-4q}\delta_q)$-invariant. Let $V_{q+1} = H_q(W_q\cup\bigcup_{n = 1}^{N_{q+1}}D_{n,q+1})$. 
For $g \in W_q\setminus W_{q-1}$ we define $c_g = c_{g,q-1}$. We also observe that $\bigcup_{q = 1}^\infty W_q = G$, and that $H_qW_q \subseteq W_{q+1}$.

Now let $h \in G$ be arbitrary and let $q_h \in \mathbb{N}$ be such that $h \in H_{q_h}$. We see that for $q \ge q_h+1$ and $N_q < n \le N_{q+1}$ we have

\begin{alignat*}{2}
    &\left|\int_{W_{q+1}^c\cup W_{q-1}}c_{hg}\overline{c_g}d\nu_n(g)\right| \le \sum_{m = q+1}^\infty\int_{W_{m+1}\setminus W_m}\left|c_{hg}\overline{c_g}\right|d\nu_n(g)+\int_{W_{q-1}}\left|c_{hg}\overline{c_g}\right|d\nu_n(g)\\
    \le&\sum_{m = q+1}^\infty2^{2m+1}\nu_n(G\setminus W_m)+2^{2q-3}\nu_n(W_{q-1}) \le \sum_{m = q+1}^\infty 2^{-2m+3}\delta_m+2^{-2q+1}\delta_{q-1} \le \delta_{q-1}.
\end{alignat*}
Next, we observe that if $T$ is a tile of $\mathcal{T}_{q+1}$ contained in $D_{n,q+1}\cap(W_{q+1}\setminus W_q)$, then $\nu_{n,T}$ is $(K_q,2^{-4q}\delta_q)$-invariant, so by Lemma \ref{KeyLemmaForAmenableGroups} we have

\begin{equation}\label{ACalculationForTheMainResult}
     \int_{T}c_{hg}\overline{c_g}d\nu_{n,T}(g) = \int_Gc_{hg}\overline{c_g}d\nu_{n,T}(g) \overset{\delta_q}{=} \int_Gc_{hg,q}\overline{c_{g,q}}d\nu_{n,T}(g) \overset{\epsilon_q}{=} \langle \tau_hf,f\rangle.
\end{equation}
Now let us suppose that $T$ is a tile of $\mathcal{T}_{q+1}$ contained in $D_{n,q+1}\cap(W_q\setminus W_{q-1})$. Since $\nu_{n,T}$ is $(K_q,2^{-4q}\delta_q)$-invariant we may apply Lemma \ref{GoodTilingLemma} to obtain a finite union of tiles of $\mathcal{T}_q$ that we denote by $D_T$ for which $\nu_{n,T}(D_T) > 1-2^{-6q+6}\delta_{q-1}$, such that if $T_0$ is a tile of $\mathcal{T}_q$ that is contained in $D_T$, then $\nu_{n,T_0} = (\nu_{n,T})_{T_0}$ is $(K_{q-1},2^{-6q+4}\delta_q)$-invariant. As in Equation \eqref{ACalculationForTheMainResult}, we have

\begin{equation}
    \left|\int_{T_0}c_{hg}\overline{c_g}d\nu_{n,T_0}(g)-\langle \tau_hf,f\rangle\right| < \delta_{q-1}+\epsilon_{q-1}.
\end{equation}
Consequently, we see that for $q > \log_2(1+||f||_2)$, we have

\begin{alignat*}{2}
    &\int_Tc_{hg}\overline{c_g}d\nu_{n,T}(g) \overset{\delta_{q-1}}{=} \int_{D_T}c_{hg}\overline{c_g}d\nu_{n,T}(g)\\
    = &\sum_{T' \in D_T}\nu_{n,T}(T')\int_{T'}c_{hg}\overline{c_g}d\nu_{n,T'}(g) \overset{\delta_{q-1}+\epsilon_{q-1}}{=} \sum_{T' \in D_T}\nu_{n,T}(T')\langle \tau_hf,f\rangle \overset{\delta_{q-1}}{=} \langle \tau_hf,f\rangle.
\end{alignat*}
Putting together the above pieces, we see that for $q \ge q_h+1+\log_2(1+||f||_2)$ and $N_q < n \le N_{q+1}$, we have

\begin{alignat*}{2}
    &\int_Gc_{hg}\overline{c_g}d\nu_n(g) \overset{\delta_{q-1}} = \int_{W_{q+1}\setminus W_q}c_{hg}\overline{c_g}d\nu_n(g)+\int_{W_q\setminus W_{q-1}}c_{hg}\overline{c_g}d\nu_n(g)\\
    \overset{\delta_{q+1}}{=}& \int_{D_{n,q+1}\cap(W_{q+1}\setminus W_q)}c_{hg}\overline{c_g}d\nu_n(g)+\int_{D_{n,q+1}\cap(W_q\setminus W_{q-1})}c_{hg}\overline{c_g}d\nu_n(g)\\
    =&\sum_{T \in D_{n,q+1}\cap(W_{q+1}\setminus W_q)}\nu_n(T)\int_Tc_{hg}\overline{c_g}d\nu_{n,T}+\sum_{T \in D_{n,q+1}\cap(W_q\setminus W_{q-1})}\nu_n(T)\int_Tc_{hg}\overline{c_g}d\nu_{n,T}\\
    \overset{3\delta_{q-1}+\epsilon_{q-1}}{=}&\sum_{T \in D_{n,q+1}\cap(W_{q+1}\setminus W_q)}\nu_n(T)\langle \tau_hf,f\rangle+\sum_{T \in D_{n,q+1}\cap(W_q\setminus W_{q-1})}\nu_n(T)\langle \tau_hf,f\rangle \overset{\delta_{q+1}}{=} \langle \tau_hf,f\rangle
\end{alignat*}
\end{proof}

Our next lemma is well known in the folklore, but we record it here for the sake of concreteness.

\begin{lemma}\label{Modulus1RepresentationForAbelianGroups}
    Let $G$ be a countably infinite abelian group and let $\nu$ be a probability measure on $\widehat{G}$. Let $S(G) \subseteq \mathbb{S}^1$ be the smallest closed set that contains the range of all characters of $G$. There exists a measure preserving system $\mathcal{X} := (X,\mathscr{B},\mu,(\tau_g)_{g \in G})$ and a measurable $f:X\rightarrow S(G)$ for which $\hat{\nu}(h) = \langle \tau_hf,f\rangle$ and $\nu(\{0\}) = \int_Xfd\mu$. Furthermore, the maximal spectral type of $\mathcal{X}$ is $\sum_{n \in \mathbb{Z}}\nu_n$, where $\nu_n(E) = \nu(\{x \in \widehat{G}\ |\ x^n \in E\})$.
\end{lemma}

\begin{proof}
    Let $X = \widehat{G}\times S(G)$, let $\mathscr{B}$ be the Borel $\sigma$-algebra, let $\tau:X\rightarrow X$ be given by $\tau_g(\chi,x) = (\chi,\chi(g)x)$, and let $\mu = \nu\times m$, where $m$ is the normalized Haar measure of the compact group $S(G)$. Let $f(\chi,x) = x$ if $\chi \neq e_{\widehat{G}}$, and $f(e_{\widehat{G}},x) = 1$. We see that

    \begin{alignat*}{2}
        &\langle \tau_hf,f\rangle = \int_{\widehat{G}}\int_{S(G)}\chi(h)dm(x)d\nu(\chi) = \int_{\widehat{G}}\chi(h)d\nu(\chi) = \hat{\nu}(h) = \phi(h)\text{, and}\\
        &\int_Xfd\mu = \int_{\widehat{G}}\int_{S(G)}f(\chi,x)dm(x)d\nu(\chi) = \int_{\widehat{G}}\mathbbm{1}_{e_{\widehat{G}}}(\chi)d\nu(\chi) = \nu(\{0\}).
    \end{alignat*}
    It only remains to show that the maximal spectral type of $\mathcal{X}$ is of the given form. Since $X$ is a compact abelian group, the characters of $X$ have a dense span in $L^2(X,\mu)$, so it suffices to show that the spectral measure of each character is some $\nu_n$. We note that $S(G)$ is either a finite set, or it is $\mathbb{T}$, so any character on $S(G)$ is of the form $x\mapsto x^s$ for some $s \in \mathbb{Z}$. Let $g \in G = \widehat{\widehat{G}}$ and $s \in \mathbb{Z}$ both be arbitrary, let $f'(\chi,x) = \chi(g)x^s$, and observe that

    \begin{equation}
        \langle \tau_hf',f'\rangle = \int_{\widehat{G}\times S(G)}\chi(g)(\chi(h)x)^s\overline{\chi(g)x^s}d\mu(\chi,x) = \int_{\widehat{G}\times S(G)}\chi(h)^sd\mu(\chi,x) = \hat{\nu}_s(h).
    \end{equation}
\end{proof}

\begin{theorem}\label{FvdCIsvdCForAmenableGroups}
    Let $G$ be a countably infinite amenable group, let $(\nu_n)_{n = 1}^\infty$ be a Reiter sequence, and let $V \subseteq G$. Items (i)-(iii) are equivalent, items (iv) and (v) are equivalent, and if $G$ is abelian, then items (i)-(v) are equivalent.
    \begin{enumerate}[(i)]
        \item For any sequence $(u_g)_{g \in G}$ of complex numbers satisfying

        \begin{alignat}{2}
            &\limsup_{n\rightarrow\infty}\int_G|u_g|^2d\nu_n(g) < \infty,\quad \sup_{h \in G}\limsup_{n\rightarrow\infty}\left|\int_G(u_{hg}-u_g)d\nu_n(g)\right| = 0\text{, and }\label{ComplexNumbersSequenceEquation1}\\
            &\lim_{n\rightarrow\infty}\int_Gu_{vg}\overline{u_g}d\nu_n(g) = 0,
        \end{alignat}
        for all $v \in V$, we have

        \begin{equation}
            \lim_{n\rightarrow\infty}\int_Gu_gd\nu_n(g) = 0.
        \end{equation}
        
        \item For any separable Hilbert space and any sequence $(\xi_g)_{g \in G} \subseteq \mathcal{H}$ of vectors satisfying

        \begin{alignat}{2}
            &\limsup_{n\rightarrow\infty}\int_G||\xi_g||^2d\nu_n(g) < \infty,\quad \sup_{h \in G}\limsup_{n\rightarrow\infty}\left|\left|\int_G(\xi_{hg}-\xi_g)d\nu_n(g)\right|\right| = 0\text{, and }\label{HilbertSpaceSequenceEquation1}\\
            &\lim_{n\rightarrow\infty}\int_G\langle\xi_{vg},\xi_g\rangle d\nu_n(g) = 0,\label{HilbertSpaceSequenceEquation2}
        \end{alignat}
        for all $v \in V$, we have

        \begin{equation}
            \lim_{n\rightarrow\infty}\left|\left|\int_G\xi_gd\nu_n(g)\right|\right| = 0.
        \end{equation}

        \item For any measure preserving system $(X,\mathscr{B},\mu,(\tau_g)_{g \in G})$ and any $f \in L^2(X,\mu)$ satisfying $\langle \tau_vf,f\rangle = 0$ for all $v \in V$, we have $\int_Xfd\mu = 0$.

        \item For any sequence $(u_g)_{g \in G} \subseteq \mathbb{S}^1$ satisfying

        \begin{equation}
            \lim_{n\rightarrow\infty}\int_Gu_{vg}\overline{u_g}d\nu_n(g) = 0,\text{ for all }v \in V,\text{ we have }\lim_{n\rightarrow\infty}\int_Gu_gd\nu_n(g) = 0.
        \end{equation}

        \item For any measure preserving system $(X,\mathscr{B},\mu,(\tau_g)_{g \in G})$ and any $f:X\rightarrow\mathbb{S}^1$ satisfying $\langle \tau_vf,f\rangle = 0$ for all $v \in V$, we have $\int_Xfd\mu = 0$.
    \end{enumerate}
\end{theorem}

\begin{proof}
We first show that (iii)$\rightarrow$(ii). Let us assume for the sake of contradiction that Equations \eqref{HilbertSpaceSequenceEquation1}-\eqref{HilbertSpaceSequenceEquation2} are satisfied, but there is some $(M_q)_{q = 1}^\infty \subseteq \mathbb{N}$ for which

\begin{equation}
    \lim_{q\rightarrow\infty}\frac{1}{M_q}\left|\left|\int_G\xi_gd\nu_{M_q}(g)\right|\right| = \epsilon > 0.
\end{equation}
Let $S_{h,q} = \int_G\xi_{h^{-1}g}d\nu_{M_q}(g)$ and let $\xi_{g,q} = \xi_g-S_{e,q}$. By replacing $(M_q)_{q = 1}^\infty$ with a subsequence, we may assume without loss of generality that

\begin{alignat*}{2}
    &\gamma_1(h) := \lim_{q\rightarrow\infty}\int_G\langle\xi_{h^{-1}g},\xi_g\rangle d\nu_{M_q}(g)\text{ and }\gamma_2(h) := \lim_{q\rightarrow\infty}\int_G\langle \xi_{h^{-1}g,q},\xi_{g,q}\rangle d\nu_{M_q}(g)
\end{alignat*}
exist for all $h \in G$. To see that $(\gamma_1(h))_{h \in G}$ is a positive definite sequence, we see that for any $g_1,\cdots,g_n \in G$ and any $c_1,\cdots,c_n \in \mathbb{C}$, we have

\begin{alignat*}{2}
&\sum_{i,j = 1}^nc_i\overline{c_j}\gamma(g_j^{-1}g_i) = \sum_{i,j = 1}^nc_i\overline{c_j}\lim_{q\rightarrow\infty}\int_G\langle \xi_{g_i^{-1}g_jg},\xi_g\rangle d\nu_{M_q}(g) = \sum_{i,j = 1}^nc_i\overline{c_j}\lim_{q\rightarrow\infty}\int_G\langle \xi_{g_i^{-1}g},\xi_{g_j^{-1}g}\rangle d\nu_{M_q}(g)\\
=&\lim_{q\rightarrow\infty}\int_G\langle\sum_{i = 1}^nc_i\xi_{g_i^{-1}g},\sum_{j = 1}^nc_j\xi_{g_j^{-1}g} d\nu_{M_q}(g) \ge 0.
\end{alignat*}
A similar calculation shows that $(\gamma_2(h))_{h \in G}$ is also a positive definite sequence.
Using the second assumption in Equation \eqref{HilbertSpaceSequenceEquation1}, we see that

\begin{alignat*}{2}
    \gamma_2(h)&= \lim_{q\rightarrow\infty}\int_G\langle \xi_{h^{-1}g}-S_{e,q},\xi_g-S_{e,q}\rangle d\nu_{M_q}(g)\\
    &= \gamma_1(h)+\lim_{q\rightarrow\infty}\left(-\int_G\langle S_{e,q},\xi_g\rangle d\nu_{M_q}(g)-\int_G\langle \xi_{h^{-1}g},S_{e,q}\rangle d\nu_{M_q}(g)+\int_G\langle S_{e,q},S_{e,q}\rangle d\nu_{M_q}(g)\right)\\
    &= \gamma_1(h)+\lim_{q\rightarrow\infty}\left(-\langle S_{e,q},S_{e,q}\rangle -\langle S_{h,q},S_{e,q}\rangle+\langle S_{e,q},S_{e,q}\rangle\right) = \gamma_1(h)-\epsilon^2.
\end{alignat*}
We now use the equivalence of items \eqref{GeneralMPSvdC} and \eqref{MeanOnAPFunctionsCharacterization1} in Theorem \ref{EquivalentCharacterizationsInCountableGroups}. Letting $M$ denote the unique invariant mean on the set $W(G)$ of weakly almost periodic functions on $G$, we see that $\gamma_1(v) = \overline{\gamma_1(v^{-1})} = 0$ for all $v \in V$, so $M(\gamma_1) = 0$. It follows that $M(\gamma_2) = -\epsilon^2 < 0$, but this contradicts the fact that $M(\phi) \ge 0$ whenever $\phi$ is a positive definite function on $G$ (see Section \ref{SectionOnMeans}).

It is clear that (ii)$\rightarrow$(i). Then fact that (i)$\rightarrow$(iii) and (iv)$\rightarrow$(v) are a consequence of Theorem \ref{RepresentationTheoremForAmenableGroups}. To see that (v)$\rightarrow$(iv), we will assume familiarity with the Stone-\v{C}ech compactification $\beta G$ of $G$, and refer the reader to \cite{AlgebraInTheSCC} for background. 
For $n \in \mathbb{N}$, let $u:G\rightarrow\mathbb{C}$ be given by $u(g) = u_g$, and let $\tilde{u}:\beta G\rightarrow\mathbb{C}$ be the unique continuous extension of $u$. 
We see that each $\nu_n$ has a unique extension to a probability measure $\tilde{\nu}_n$ on $\beta G$. Let $\mu$ be any probability measure on $(\beta G,\mathscr{A})$ with $\mathscr{A}$ the Borel $\sigma$-algebra that is a weak$^*$ limit of the sequence $\left\{\tilde{\nu}_n\right\}_{n = 1}^\infty$, and let $\left\{\tilde{\nu}_{M_q}\right\}_{q = 1}^\infty$ be a subsequence converging to $\mu$.
Let $\tau_g:\beta G\rightarrow \beta G$ be given by $\tau_g(p) = g^{-1}\cdot p$,\footnote{It is worth noting that we are using different notation than in \cite{AlgebraInTheSCC} since we are assuming that $g^{-1}\cdot p$ is continuous with respect to the variable $p$ instead of the variable $g$. The necessity to do so stems from the fact that we chose to work with left-asymptotically invariant sequences of probability rather than right.}, hence measurable. 
Letting $\mathscr{B}$ be the countably generated $\sigma$-algebra of $\tilde{u}$ and $(\tau_g)_{g \in G}$, we see that $(\beta G,\mathscr{B},(\tau_g)_{g \in G},\mu)$ is isomorphic to a measure preserving system on a standard probability space. Lastly, we see that

\begin{equation}
    \langle \tau_v\tilde{u},\tilde{u}\rangle  = \overline{\langle \tau_{v^{-1}}\tilde{u},\tilde{u}\rangle} = \lim_{q\rightarrow\infty}\int_G\overline{u_{vg}}u_gd\nu_{M_q}\text{ and }\int_{\beta G}\tilde{u}d\mu = \lim_{q\rightarrow\infty}\int_Gu_gd\nu_{M_q}(g).
\end{equation}

It is clear that (iii)$\rightarrow$(v). Now let us show that (v)$\rightarrow$(iii) when $G$ is abelian. We see that if $(X,\mathscr{B},\mu,(\tau_g)_{g \in G})$ is a measure preserving system and $f \in L^2(X,\mu)$ is normalized so that $||f||_2 = 1$, then $\phi(g) = \langle \tau_gf,f\rangle$ is a positive definite sequence with $\phi(e) = 1$, so there exists a probability measure $\nu$ on $\widehat{G}$ for which $\hat{\nu}(g) = \langle \tau_gf,f\rangle$ and $\nu(\{0\}) = ||P_If||_2^2 \ge \left|\int_Xfd\mu\right|^2$. We may use Lemma \ref{Modulus1RepresentationForAbelianGroups} to obtain a measure preserving system $(Y,\mathscr{A},\mu',(S_g)_{g \in G})$ and a measurable $f':Y\rightarrow S(G)$ satisfying $\phi(g) = \langle S_gf',f'\rangle$ and $\nu(\{0\}) = \int_Yf'd\mu'$.
\end{proof}

\begin{remark}\label{RemarkAboutCompletingTheListOfEquivalences}
    Now let us consider an example to show why we need the second condition in Equations \eqref{ComplexNumbersSequenceEquation1} and \eqref{HilbertSpaceSequenceEquation1} in Theorem \ref{FvdCIsvdCForAmenableGroups} despite not needing these conditions in Theorem \ref{EquivalentCharacterizationsOfvdC}. 
    Let $G = \mathbb{Z}$ and consider the F\o lner sequence $F_n = [n^3,n^3+2n]$. For $m \in [n^3,n^3+n]$ let $u_m = 1$, let $u_{n^3+2n} = 1$, for $m \in [n^3+2n+1,n^3+3n]$ let $u_m = -n$, and let $u_m = 0$ for all other values of $m$. 
    We see that

    \begin{equation*}
        \lim_{n\rightarrow\infty}\frac{1}{|F_n|}\sum_{m \in F_n}|u_m|^2 = \lim_{n\rightarrow\infty}\frac{1}{|F_n|}\sum_{m \in F_n}u_m = \frac{1}{2}\text{ and }\lim_{n\rightarrow\infty}\frac{1}{|F_n|}\sum_{m \in F_n}u_{m+h}\overline{u_m} = 0\text{ for all }h \in \mathbb{N}.
    \end{equation*}

    Furthermore, in Theorem \ref{FvdCIsvdCForAmenableGroups}, we would like to show that (i)-(iv) are equivalent for any amenable group. This would follow from our proof provided the following questions has a positive answer for all amenable $G$.
\end{remark}

\begin{question}\label{MainConjecture}
    Let $G$ be a countable group and let $\phi:G\rightarrow\mathbb{C}$ be a positive definite sequence for which $\phi(e) = 1$. Does there exists a measure preserving system $(X,\mathscr{B},\mu,(\tau_g)_{g \in G})$ and a measurable $f:X\rightarrow\mathbb{S}^1$ for which the following holds:
    \begin{enumerate}[(i)]
        \item $\phi(h) = \langle \tau_hf,f\rangle$ for all $h \in G$.
        \item $\int_Xfd\mu = 0$ if and only if $f$ is orthogonal to the subspace of $L^2(X,\mu)$ of $T$-invariant functions.
    \end{enumerate} 
\end{question}
\section{Appendix: Properties of sets of operatorial recurrence}\label{Section:Appendix}
We begin with a list of the equivalent characterizations of vdC sets/sets of operatorial recurrence that were omitted from Theorem \ref{EquivalentCharacterizationsOfvdC} in Theorem \ref{EquivalentCharacterizationsOfvdCAppendix}. We then generalize most of these equivalences to the setting of countably infinite groups in Theorem \ref{EquivalentCharacterizationsInCountableGroups}, and some of them only to the setting of countably infinite abelian groups in Theorem \ref{SpectralCharacterizationsForAbelianGroups}. Lastly, in Theorem \ref{PropertiesOfSetsOfOperatorialRecurrence}, we list properties of sets of operatorial recurrence that follow from the work of Rodr\'iguez \cite{SaulsvdC}.

We mention that an important result in the study of sets of operatorial recurrence in $\mathbb{N}$ is Bourgain's construction \cite{vdCStrongerThanRecurrence} (see also \cite{RefiningBourgain}) of a set of measurable recurrence that is not a set of operatorial recurrence.\footnote{Bourgain used the term vdC set in his work.} While we do not study this construction here, we believe that our many equivalent formulations of sets of operatorial recurrence may help generalize Bourgain's construction to a larger class of groups, and shed more light on the difference between measurable and operatorial recurrence.

\begin{theorem}\label{EquivalentCharacterizationsOfvdCAppendix}
    For $V \subseteq \mathbb{N}$, the following are equivalent:
    \begin{enumerate}[(i)]
        \item\label{UniformDistributionvdCAppendix} $V$ is a vdC set.
        
        \item\label{OperatorialRecurrenceAppendix} $V$ is a set of operatorial recurrence.

        \item\label{pointmassat0vdc} For any probability measure $\mu$ on $[0,1]$ satisfying $\hat{\mu}(v) = 0$ for all $v \in V$, we have $\mu(\{0\}) = 0$.

        \item\label{continuousmeasurevdc} Any probability measure $\mu$ on $[0,1]$ satisfying $\hat{\mu}(v) = 0$ for all $v \in V$ must be continuous.

        \item\label{measure/recurrencevdc} Any probability measure $\mu$ on $[0,1]$ satisfying $\sum_{v \in V}|\hat{\mu}(v)| < \infty$ must be continuous. 

        \item\label{Modulus1Functions} For any measure preserving system $(X,\mathscr{B},\mu,\tau)$ and any measurable $f:X\rightarrow\mathbb{S}^1$ satisfying $\langle \tau^vf,f\rangle = 0$ for all $v \in V$, we have $\int_Xfd\mu = 0$.

        \item\label{ErgodicMPSvdC} For any ergodic measure preserving system $(X,\mathscr{B},\mu,\tau)$ and any measurable $f \in L^2(X,\mu)$ satisfying $\langle \tau^vf,f\rangle = 0$ for all $v \in V$, we have $\int_Xfd\mu = 0$.

        \item\label{positivedefinitevdc} For any $\epsilon > 0$, there exists a finite, positive definite sequence $(a_n)_{n \in \mathbb{Z}}$ supported on $V\cup(-V)\cup\{0\}$ satisfying

        \begin{equation}
            \sum_{n \in \mathbb{Z}}a_n = 1\text{ and }a_0 < \epsilon.
        \end{equation}

        \item\label{WAPMeanCharacterization} Let $M$ denote the unique invariant mean on the set weakly almost periodic functions on $\mathbb{Z}$. If $\phi:\mathbb{Z}\rightarrow\mathbb{C}$ is a positive definite function for which $\phi(v) = 0$ for all $v \in V$, then $M(\phi) = 0$.

        \item\label{trigonometricpolynomialvdc} For any $\epsilon > 0$, there exists a trigonometric polynomial $P:[0,1]\rightarrow[-\epsilon,\infty)$ of the form

        \begin{equation}\label{trigonometricpolynomialequation}
            P(x) = \sum_{v \in V\cup(-V)}a_ve(vx)
        \end{equation}
        satisfying $P(0) = 1$.
    \end{enumerate}
\end{theorem}

The equivalence of \eqref{UniformDistributionvdCAppendix} and \eqref{OperatorialRecurrenceAppendix} is part of Theorem \ref{EquivalentCharacterizationsOfvdC}. The equivalence of \eqref{UniformDistributionvdCAppendix} and \eqref{pointmassat0vdc} is implicitly alluded to in the work of Kamae and Mendes-France \cite{KMFvdC}, and it was proven that \eqref{trigonometricpolynomialvdc}$\Rightarrow$\eqref{UniformDistributionvdCAppendix}. The equivalence of \eqref{UniformDistributionvdC}, \eqref{pointmassat0vdc}, and \eqref{trigonometricpolynomialvdc} was proven in the work of Ruzsa \cite{RuzsavanderCorput}. The equivalence of \eqref{UniformDistributionvdC}, \eqref{positivedefinitevdc}, and \eqref{measure/recurrencevdc} is due to Bergelson and Lesigne \cite{vanderCorputSetsInZ^d}. The equivalence of \eqref{pointmassat0vdc} and \eqref{continuousmeasurevdc} was known in the folklore for a long time, as many older papers also refer to vdC sets as FC$^+$ sets, with FC$^+$ being the abbreviation of ``Forces continuity of positive measures". The characterizations given by \eqref{Modulus1Functions}, \eqref{ErgodicMPSvdC}, and \eqref{WAPMeanCharacterization} are results of this paper.

Theorem \ref{FvdCIsvdCForAmenableGroups}  and Remark \ref{RemarkAboutCompletingTheListOfEquivalences} is our attempt to generalize Theorem \ref{EquivalentCharacterizationsOfvdC}\eqref{boundedcomplexvdC}-\eqref{MPSvdC} to the setting of countably infinite amenable groups. The work of Rodr\'iguez \cite{SaulsvdC} generalizes Theorem \ref{EquivalentCharacterizationsOfvdC}(i)-(ii) to the setting of countably infinite amenable groups. It is worth noting that if our group $G$ is not amenable we cannot easily talk about vdC sets and the equivalent characerizations that involve F\o lner sequences. We focus the rest of the disucssion on equivalent characterizations of sets of operatorial recurrence on general countably infinite groups $G$.

\begin{theorem}\label{EquivalentCharacterizationsInCountableGroups}
Let $G$ be a countable discrete group and let $M$ denote the unique mean on the set $W(G)$ of weakly almost periodic functions on $G$. For a set $V \subseteq G\setminus\{e\}$, the following are equivalent:
    \begin{enumerate}[(i)]
        \item \label{item:unitary_inv2} $V$ is a \textbf{set of operatorial recurrence}, i.e., for every unitary representation $\pi$ of $G$ on a Hilbert space $\mathcal{H}_\pi$, and every vector $\xi \in \mathcal{H}_{\pi}$, if $\langle \pi (v)\xi ,\xi \rangle =0$ for all $v\in V$, then $\xi$ is orthogonal to the subspace of $\pi (G)$-invariant vectors.  

        \item \label{item:unitary_approx_inv}
        For every $\epsilon >0$ and every finite set $H \subseteq G$ there is some $\delta >0$ and $F\subseteq V$ finite such that for every unitary representation $\pi$ of $G$ on $\mathcal{H}_\pi$, and every unit vector $\xi \in \mathcal{H}_{\pi}$, if $\sup _{v\in F}|\langle \pi (v)\xi ,\xi \rangle | <\delta$ then $|\langle \xi ,\eta \rangle |<\epsilon$ for every $(\pi (H),\delta )$-invariant unit vector $\eta \in \mathcal{H}_{\pi}$.

        \item\label{GeneralMPSvdC} For any measure preserving system $(X,\mathscr{B},\mu,(\tau_g)_{g \in G})$ and any $f \in L^2(X,\mu)$ satisfying $\langle \tau_vf,f\rangle\allowbreak = 0$ for all $v \in V$, we have $\int_Xfd\mu = 0$.

        \item\label{GeneralErgodicMPSvdC} For any ergodic measure preserving system $(X,\mathscr{B},\mu,(\tau_g)_{g \in G})$ and any $f \in L^2(X,\mu)$ satisfying $\langle \tau_vf,f\rangle = 0$ for all $v \in V$, we have $\int_Xfd\mu = 0$.
        
        \item\label{GeneralTrigonometricPolynomialvdC} For any unitary representation $U$ of $G$ on a Hilbert space $\mathcal{H}$ and any $\epsilon > 0$, there exists

        \begin{equation*}
            P \in B(V) := \left\{\sum_{g \in G}c_gU_g\ |\ (c_g)_{g \in G}\text{ has finite support contained in }V\cup V^{-1}\text{ and }\sum_{g \in G}c_g = 1\right\},
        \end{equation*}
        such that $P = P^*$ and $P+\epsilon$ is a positive operator.
        \item\label{Generalpositivedefinitevdc} For any $\epsilon > 0$, there exists a positive definite sequence $(a_g)_{g \in G}$ with finite support contained in $V\cup V^{-1}\cup\{e\}$ satisfying

        \begin{equation}
            \sum_{g \in G}a_g = 1\text{ and }|a_e| < \epsilon.
        \end{equation}

        \item\label{item:GeneralizedUnitary_inv2} For every unitary representation $\pi$ of $G$ on $\mathcal{H}_\pi$, and every vector $\xi \in \mathcal{H}_{\pi}$, if 
        
        \begin{equation}
            \sum_{v \in V}\left|\langle \pi(v)\xi,\xi\rangle\right| < \infty,
        \end{equation}
        then $\xi$ is orthogonal to the subspace of $\pi (G)$-invariant vectors.   

        \item\label{item:GeneralizedUnitary_invp} For every unitary representation $\pi$ of $G$ on $\mathcal{H}_\pi$, and every vector $\xi \in \mathcal{H}_{\pi}$, if there exists $p \in \mathbb{N}$ for which
        
        \begin{equation}
            \sum_{v \in V}\left|\langle \pi(v)\xi,\xi\rangle\right|^p < \infty,
        \end{equation}
        then $\xi$ is orthogonal to the subspace of $\pi (G)$-invariant vectors.   
        
        \item\label{item:GeneralizedWMUnitary_inv2} For every unitary representation $\pi$ of $G$ on $\mathcal{H}_\pi$, and every vector $\xi \in \mathcal{H}_{\pi}$, if there exists $p \in \mathbb{N}$ for which
        
        \begin{equation}
            \sum_{v \in V}\left|\langle \pi(v)\xi,\xi\rangle\right|^p < \infty,
        \end{equation}
        then $\xi$ is orthogonal to the closed subspace spanned by the finite dimensional subrepresentations of $\pi$.

        \item\label{MeanOnAPFunctionsCharacterization1} If $\phi \in \mathbf{P}(G)$ is such that $\phi(v) = 0$ for all $v \in V$, then $M(\phi) = 0$.

        \item\label{MeanOnAPFunctionsCharacterization2} If $\phi \in \mathbf{P}(G)$ is such that $\sum_{v \in V}|\phi(v)|^p < \infty$ for some $p \in \mathbb{N}$, then $M(|\phi|) = 0$.
    \end{enumerate}
\end{theorem}

\begin{proof}
    We first show that \eqref{item:unitary_inv2}$\Rightarrow$\eqref{GeneralTrigonometricPolynomialvdC}. 
    Let $A$ denote the set of all nonnegative bounded linear operators on the (complex) Hilbert space $\mathcal{H}$, let $B(\mathcal{H})$ denote the Banach space of self-adjoint bounded linear operators on $\mathcal{H}$, and let $B_{\mathbb{R}}(\mathcal{H})$ denote $B(\mathcal{H})$ viewed as a real-Banach space.
    Observe that $A$ is a closed convex set with nonempty interior in the $B_{\mathbb{R}}(\mathcal{H})$.
    Let $B = B(V)$, and let us assume for the sake of contradiction that there exists $\epsilon > 0$ for which $(B+\epsilon)\cap A = \emptyset$. 
    Since $(B+\epsilon)\cap B_{\mathbb{R}}(\mathcal{H}) = \{b+\epsilon\ |\ b \in B\text{ and }b = b^*\}$ is also a convex set, the Hahn-Banach separation theorem gives us a real-valued continuous linear functional $f$ on $B_{\mathbb{R}}(\mathcal{H})$, for which $r_A := \inf_{a \in A}f(a) \ge \sup\{f(b+\epsilon)\ |\ b \in B\text{ and }b = b^*\}$. We note that for any $a \in A$ and $\lambda \in \mathbb{R}^+$, we have $\lambda a \in A$, hence $r_A = 0$. It follows that $f$ is a positive linear functional, so we may assume without loss of generality that $||f|| = 1$. We extend $f$ by linearity to be a complex-valued functional on the Banach space $B(\mathcal{H})$ of all bounded linear operators on $\mathcal{H}$. Now we observe that for $\lambda \in \mathbb{R}$, $v \in V$, and $b \in B$, we have $b+\lambda i(U_v-U_{v^{-1}}) \in B$, hence

    \begin{equation}
        0 \ge f(b+\lambda i(U_v-U_{v^{-1}})) = f(b)+\lambda f(i(U_v-U_{v^{-1}})).
    \end{equation}
    Since $\lambda \in \mathbb{R}$ was arbitrary, we conclude that for all $v \in V$ we have 
    
    \begin{equation}\label{SameValueEquation}
        f(i(U_v-U_{v^{-1}})) = 0 \Rightarrow f(U_v) = f(U_{v^{-1}}).
    \end{equation}
    Similarly, we see that for any $\lambda \in \mathbb{R}$, $v_1,v_2 \in V$, and $b \in B$, we have $b+\lambda(U_{v_1}+U_{v_1^{-1}}-U_{v_2}-U_{v_2^{-1}}) \in B$, hence

    \begin{equation}
        0 \ge f\left(b+\lambda\left(U_{v_1}+U_{v_1^{-1}}-U_{v_2}-U_{v_2^{-1}}\right)\right) = f(b)+\lambda f\left(U_{v_1}+U_{v_1^{-1}}-U_{v_2}-U_{v_2^{-1}}\right).
    \end{equation}
    Since $\lambda \in \mathbb{R}$ and $v_1,v_2 \in V$ were all arbitrary, we see that 

    \begin{equation}
        f\left(U_{v_1}+U_{v_1^{-1}}\right) = f\left(U_{v_2}+U_{v_2^{-1}}\right) = 2r.
    \end{equation}
    Combining this with Equation \eqref{SameValueEquation}, we see that for any $v \in V\cup V^{-1}$ we have $f(U_v) = r$. Since $\frac{1}{2}(U_v+U_{v^{-1}}) \in B$, we see that $r \le 0$.
    
    We now claim that the sequence $(f(U_g))_{g \in G}$ is a positive definite sequence. To prove the claim, let $c_1,\cdots,c_n \in \mathbb{C}$ and $g_1,\cdots,g_n \in G$ be arbitrary, and let $V = \sum_{i,j = 1}^nc_i\overline{c_j}U_{g_j^{-1}g_i}$. We want to show that $f(V) \ge 0$, so it suffices to show that $V$ is a positive operator. To this end, let $\xi \in \mathcal{H}$ be arbitrary, and observe that

    \begin{equation}
        \langle V\xi,\xi\rangle = \langle\sum_{i,j = 1}^nc_i\overline{c_j}U_{g_j^{-1}g_i}\xi,\xi\rangle = \langle \sum_{i = 1}^nc_iU_{g_i}\xi,\sum_{i = 1}^nc_iU_{g_i}\xi\rangle \ge 0.
    \end{equation}
    
     Now that we have proven the claim, we use the GNS-construction to create a representation $\pi$ of $G$ on $\mathcal{H}'$ and a cyclic vector $\eta \in \mathcal{H}'$ for which $\langle \pi_g\eta,\eta\rangle_{\mathcal{H}'} = f(U_g)$ for all $g \in G$. Now let $\mathcal{H}'' = \mathcal{H}'\oplus\mathbb{C}$, let $\xi = (\eta,\sqrt{-r})$, and let $\pi_g' = \pi_g\oplus\text{Id}$. We see that for every $v \in V$ we have

    \begin{equation}
        \langle \pi_v'\xi,\xi\rangle_{\mathcal{H}''} = \langle \pi_v\eta,\eta\rangle_{\mathcal{H}'}+r = 0.
    \end{equation}
    Condition \eqref{item:unitary_inv2} tells us that $\xi$ is orthogonal to subspace of $\pi'(G)$-invariant vectors, which yields the desired contradiction.

    We now show that \eqref{GeneralTrigonometricPolynomialvdC}$\Rightarrow$\eqref{Generalpositivedefinitevdc}.
    Let $R$ denote the right regular representation of $G$ on $L^2(G,\lambda)$, where $\lambda$ is the counting measure.
    Let $\epsilon > 0$ be arbitrary and let $P = \sum_{h \in V\cup V^{-1}}c_hR_h$ be such that $P = P^*$, $P+\epsilon$ is positive, $(c_h)_{h \in V\cup V^{-1}}$ is finitely supported, and $\sum_{h \in V\cup V^{-1}}c_h = 1$.
    Since $P^* = \sum_{h \in V\cup V^{-1}}\overline{c_h}R_h^* = \sum_{h \in V\cup V^{-1}}\overline{c_h}R_{h^{-1}}$, we see that $c_{h^{-1}} = \overline{c_h}$ for all $h \in G$.
    Let $c_{e} = \epsilon$ and $c_g = 0$ for $g \notin V\cup V^{-1}\cup\{e\}$. 
    We will show that $(c_g)_{g \in G}$ is a positive definite sequence. 
    To this end, let $(\delta_g)_{g \in G}$ be the standard bases for $L^2(G,\lambda)$, i.e., $\delta_g(g) = 1$, and $\delta_g(h) = 0$ for all $h \neq g$.  
    We observe that $R_h\delta_g = \delta_{gh^{-1}}$.
    Let $(z_g)_{g \in G}$ be any finitely supported sequence of complex numbers, let $\xi = \sum_{g \in G}z_g\delta_g$, and observe that

    \begin{alignat*}{2}
        &\sum_{g,h \in G}z_g\overline{z_h}c_{h^{-1}g} = \left\langle\sum_{g \in G}z_g\delta_g,\sum_{g \in G}\left(\sum_{h \in G}\overline{c_{h^{-1}g}}z_h\right)\delta_g\right\rangle = \left\langle\sum_{g \in G}z_g\delta_g,\sum_{g \in G}\left(\sum_{h \in G}c_{g^{-1}h}z_h\right)\delta_g\right\rangle\\
        =& \left\langle\sum_{g \in G}z_g\delta_g,\sum_{g \in G}\left(\sum_{h \in G}c_{g^{-1}h}z_hR_{g^{-1}h}\delta_h\right)\right\rangle = \left\langle\sum_{g \in G}z_g\delta_g,\sum_{h \in G}\left(\sum_{g \in G}c_{g^{-1}h}R_{g^{-1}h}\right)z_h\delta_h\right\rangle\\
        =& \langle\xi,(P+\epsilon)\xi\rangle \ge 0.
    \end{alignat*}
    Since $\sum_{g \in G}c_g = 1+\epsilon$, we see that the desired positive definite sequence $(a_g)_{g \in G}$ is given by $a_g = \frac{1}{1+\epsilon}c_g$.

    Next, we show that \eqref{Generalpositivedefinitevdc}$\Rightarrow$\eqref{GeneralTrigonometricPolynomialvdC}. Let $\epsilon > 0$ be arbitrary, let $\epsilon' = \frac{\epsilon}{1+\epsilon}$, and observe that for $x \in (0,\epsilon')$ we have $\frac
    {x}{(1-x)} < \epsilon$. 
    Let $(a_g)_{g \in G}$ be a positive definite sequence with a finite support contained in $V\cup V^{-1}\cup\{e\}$ satisfying $\sum_{g \in G}a_g = 1$ and $|a_e| < \epsilon'$. 
    Let $P = \frac{1}{1-|a_e|}\sum_{g \in V\cup V^{-1}}a_gU_g$. 
    Since $(a_g)_{g \in G}$ is positive definite, we see that $a_g = \overline{a_{g^{-1}}}$, so $P = P^*$.
    Since $P+\epsilon > P' = \frac{1}{1-|a_e|}\sum_{g \in V\cup V^{-1}\cup\{e\}}a_gU_g$, it suffices to show that $P'$ is a positive operator. 
    To this end, we see that $f \in L^2(G,\lambda)$ given by $f(g) = a_g$ is a continuous positive definite function, so using \cite[Theorem 13.8.6]{C(star)AlgebraByDixmier} we pick a continuous positive definite function $\psi \in L^2(G,\lambda)$ for which $f = \psi*\psi$ and $\psi = \tilde{\psi}$, where $*$ denotes convolution and $\tilde{F}(g) := \overline{F(g^{-1})}$. 
    Letting $\Psi = \sum_{g \in G}\psi(g)U_g$, we see that $\Psi^* = \Psi$, and $(1-\epsilon')P' = \Psi\Psi = \Psi\Psi^*$, so $P'$ is a positive operator.

    We now show that \eqref{GeneralTrigonometricPolynomialvdC}$\Rightarrow$\eqref{item:GeneralizedUnitary_inv2}. Let $\xi_I$ denote the projection of $\xi$ onto the subspace of $\pi(G)$-invariant vectors, and let $\xi = \xi_I+\xi'$. Let $\epsilon > 0$ be arbitrary, and let $P = \sum_{g \in V\cup V^{-1}}c_g\pi_g \in B(V)$ be such that $P+\epsilon$ is a positive operator. Letting $c_{e} = \epsilon$ and $c_g = 0$ for $g \notin V\cup V^{-1}\cup\{e\}$, we see that $(c_g)_{g \in G}$ is a positive definite sequence, so for all $g \in G$ we have $|c_g| \le |c_{e}| = \epsilon$. We now see that

    \begin{alignat*}{2}
        &||\xi_I||^2 = \langle P\xi_I,\xi_I\rangle < \langle (P+\epsilon)\xi_I,\xi_I\rangle \le \langle (P+\epsilon)\xi,\xi\rangle = \epsilon||\xi||^2+\sum_{g \in V\cup V^{-1}}c_g\langle \pi_g\xi,\xi\rangle\\
        \le & \epsilon||\xi||^2+\sum_{g \in V\cup V^{-1}}|c_g|\cdot\left|\langle \pi_g\xi,\xi\rangle\right| \le \epsilon||\xi||^2+\epsilon\sum_{g \in V\cup V^{-1}}|\langle\pi_g\xi,\xi\rangle|.
    \end{alignat*}
    Since $\epsilon > 0$ was arbitrary, we see that $||\xi_I||^2 = 0$.

    We now show that \eqref{item:GeneralizedUnitary_inv2}$\rightarrow$\eqref{item:GeneralizedUnitary_invp}. 
    Let $\pi^p$ be the tensor product of $p$ copies of $\pi$ acting on $\otimes_{i = 1}^p\mathcal{H}$. We see that $\xi^p := \otimes_{i = 1}^p\xi \in \otimes_{i = 1}^p\mathcal{H}$ satisfies

    \begin{equation}
        \sum_{v \in V}|\langle \pi^{p_0}(v)\xi^p,\xi^p\rangle| = \sum_{v \in V}|\langle\pi(v)\xi,\xi\rangle|^p < \infty,
    \end{equation}
    so $\xi^p$ is orthogonal to the space of $\pi^p$-invariant vectors, hence $\xi$ is orthogonal to the space of $\pi$-invariant vectors.

    It is clear that \eqref{item:GeneralizedWMUnitary_inv2}$\rightarrow$\eqref{item:unitary_inv2}, so we proceed to show that \eqref{item:GeneralizedUnitary_invp}$\rightarrow$\eqref{item:GeneralizedWMUnitary_inv2}. 
    Assume that \eqref{item:GeneralizedUnitary_invp} holds, and suppose that $\pi$ is a unitary representation of $G$, and $\xi\in \mathcal{H}=\mathcal{H}_{\pi}$ and $p\geq 1$ are such that $\sum_{v\in V}|\langle \pi _v \xi , \xi \rangle |^p<\infty$.
    Let $\bar{\pi}$ be the conjugate representation of $\pi$ on the conjugate Hilbert space $\bar{\mathcal{H}} = \{ \bar{\eta} : \eta \in \mathcal{H} \}$ of $\mathcal{H}$, i.e., scalar multiplication in $\bar{H}$ is defined by $c\bar{\eta}= \overline{\bar{c}\eta}$, the inner product is given by $\langle \bar{\eta}_0,\bar{\eta}_1\rangle =\langle \eta _1 , \eta _0 \rangle$, and $\bar{\pi}$ is defined by $\bar{\pi}_g\bar{\eta}= \overline{\pi _g \eta}$.
    Let $\mathrm{HS}(\mathcal{H})$ be the Hilbert space of all Hilbert-Schmidt operators on $\mathcal{H}$, and let $\sigma$ be the unitary representation on $\mathrm{HS}(\mathcal{H})$ given by $\sigma _g(T) := \pi _gT\pi _g ^{-1}$ for $T\in \mathrm{HS}(\mathcal{H})$ and $g\in G$.  
    Then the representations $\pi\otimes \bar{\pi}$ and $\sigma$ are isomorphic via the map $\mathcal{H}\otimes \bar{\mathcal{H}} \to \mathrm{HS}(\mathcal{H})$, $\zeta\mapsto \tau_{\zeta}$, determined by $\langle \tau_{\zeta} \eta _0, \eta _1 \rangle := \langle \zeta , \eta _1\otimes \bar{\eta} _0\rangle$, for $\zeta\in \mathcal{H}\otimes\bar{\mathcal{H}}$ and $\eta _0, \eta _1 \in \mathcal{H}$. 
    We have
    \[
    \sum _{v\in V}|\langle \sigma _v(\tau_{\xi\otimes \bar{\xi}}), \tau_{\xi\otimes \bar{\xi}}\rangle |^p = \sum _{v\in V}|\langle (\pi _v \xi )  \otimes \overline{(\pi _v \xi )}, \xi\otimes \bar{\xi} \rangle | ^p = \sum _{v\in V}|\langle \pi _v \xi , \xi \rangle | ^{2p}  < \infty ,
    \]
    so the assumption that \eqref{item:GeneralizedUnitary_invp} holds lets us deduce that $\tau_{\xi\otimes \bar{\xi}}$ is orthogonal in $\mathrm{HS}(\mathcal{H})$ to the subspace of all $\sigma$-invariant vectors. 
    In particular, given a finite dimensional $\pi$-invariant subspace $\mathcal{K}$ of $\mathcal{H}$, it follows that $\tau_{\xi\otimes \bar{\xi}}$ is orthogonal to the orthogonal projection $P_{\mathcal{K}}$ to $\mathcal{K}$.
    Taking an orthonormal basis $B_{\mathcal{K}}$ for $\mathcal{K}$ and extending it to an orthonormal basis $B$ for $\mathcal{H}$, we compute  
    \[
0=\langle \tau_{\xi\otimes \bar{\xi}} , P_{\mathcal{K}}\rangle = \sum _{e,f\in B}\langle \tau_{\xi\otimes \bar{\xi}}e,f\rangle \overline{\langle P_{\mathcal{K}}e,f\rangle} = \sum _{e\in B_{\mathcal{K}}}\langle \tau_{\xi\otimes \bar{\xi}}e, e\rangle = \sum _{e\in B_{\mathcal{K}}}|\langle \xi , e \rangle | ^2 = \| P_{\mathcal{K}}(\xi ) \| ^2,    
    \]
which shows that $\xi$ is orthogonal to $\mathcal{K}$.

    It is clear that \eqref{item:unitary_approx_inv}$\rightarrow$\eqref{item:unitary_inv2}, so let us now show that \eqref{item:unitary_inv2}$\rightarrow$\eqref{item:unitary_approx_inv}.
    Let $(v_n)_{n = 1}^\infty$ be an enumeration of the elements of $V$, and for $n \in \mathbb{N}$ let $F_n = (v_m)_{m = 1}^n$.
    Let $(H_n)_{n = 1}^\infty$ be an exhaustion of $G$ by finite sets.
    Let us assume for the sake of contradiction that there exists $\epsilon > 0$ such that for all $n \in \mathbb{N}$ there exists a unitary representation $\pi_n$ of $G$ on a Hilbert space $\mathcal{H}_n$ and a unit vector $\xi_n \in \mathcal{H}_n$ for which $\sup_{v \in F_n}|\pi_n(v)\xi_n,\xi_n\rangle| < \delta$ and $|\langle\xi_n,\eta_n\rangle| \ge \epsilon$ for some $(\pi_n(H_n),\frac{1}{n})$-invariant unit vector $\eta_n \in \mathcal{H}_n$.
    Let $p \in \beta\mathbb{N}^*$ be a nonprincipal ultrafilter and let $\mathcal{H} := \prod_p\mathcal{H}_n$ denote the ultraproduct of $(\mathcal{H}_n)_{n = 1}^\infty$ with respect to $p$.
    Let $\xi,\eta \in \mathcal{H}$ be the unit vectors corresponding to the equivalence classes of $(\xi_n)_{n = 1}^\infty$ and $(\eta_n)_{n = 1}^\infty$ respectively, and let $\pi$ be the unitary representation of $G$ on $\mathcal{H}$ given by $\pi(g)(x_n)_{n = 1}^\infty = (\pi_n(g)x_n)_{n = 1}^\infty$.
    We see that for $g \in G$ we have

    \begin{equation}
        \langle\pi(g)\eta,\eta\rangle_{\mathcal{H}} = p-\lim_{n\rightarrow\infty}\langle\pi_n(g)\eta_n,\eta_n\rangle_{\mathcal{H}_n} = 1,
    \end{equation}
    so $\eta$ is a $\pi(G)$-invariant vector. 
    We also see that for $v \in V$ we have

    \begin{equation}
        \langle\pi(v)\xi,\xi\rangle_{\mathcal{H}} = p-\lim_{n\rightarrow\infty}\langle\pi_n(v)\xi_n,\xi_n\rangle_{\mathcal{H}_n} = 0,
    \end{equation}
    so the desired contradiction now follows from the observation that

    \begin{equation}
        |\langle\xi,\eta\rangle_{\mathcal{H}}| = p-\lim_{n\rightarrow\infty}|\langle\xi_n,\eta_n\rangle_{\mathcal{H}_n}| \ge \epsilon.
    \end{equation}
    
    Now we show that \eqref{item:unitary_inv2} is equivalent to \eqref{MeanOnAPFunctionsCharacterization1} and that \eqref{item:GeneralizedWMUnitary_inv2} is equivalent to \eqref{MeanOnAPFunctionsCharacterization2}. Let $P_I:\mathcal{H}\rightarrow\mathcal{H}$ denote the orthogonal projection onto the space of $\pi$-invariant vectors.
    To this end, we recall that a function $\phi:G\rightarrow\mathbb{C}$ is positive definite if and only if there exists a unitary representation $\pi$ of $G$ on a Hilbert space $\mathcal{H}$ and a cyclic vector $\xi$ such that $\phi(g) = \langle \pi(g)\xi,\xi\rangle$. 
    The desired result follows from the observation that $M(\phi) = 0$ if and only if $||P_I\xi|| = 0$, and $M(|\phi|) = 0$ if and only if $\pi$ has no finite dimensional subrepresentations (see Section \ref{SectionOnMeans}).

    It is clear that \eqref{item:unitary_inv2}$\rightarrow$\eqref{GeneralMPSvdC}$\rightarrow$\eqref{GeneralErgodicMPSvdC}, so we proceed to show that \eqref{GeneralErgodicMPSvdC}$\rightarrow$\eqref{item:unitary_inv2}. 
    Since $\phi(g) = \langle \pi(g)\xi,\xi\rangle$ is a positive definite sequence, we use Theorem \ref{ErgodicRepresentation} to construct an ergodic m.p.s. $(X,\mathscr{B},\mu,(\tau_g)_{g \in G})$ and a $f \in L^2(X,\mu)$ for which $\phi(g) = \langle \tau_gf,f\rangle$. We observe that $\left(\int_Xfd\mu\right)^2 = M(\phi) = ||P_I\xi||^2$. Since $0 = \langle \pi(v)\xi,\xi\rangle = \langle \tau_vf,f\rangle$ for all $v \in V$, we see that $0 = \int_Xfd\mu = ||P_I\xi||$.
\end{proof}

\begin{theorem}\label{SpectralCharacterizationsForAbelianGroups}
    Let $G$ be a countably infinite abelian group. For $V \subseteq G$ the following are equivalent:
    \begin{enumerate}[(i)]
        \item\label{AbelianOperatorialRecurrence} $V$ is a set of operatorial recurrence.

        \item\label{AbelianSpectralvdC} For any probability measure $\mu$ on $\widehat{G}$ satisfying $\hat{\mu}(v) = 0$ for all $v \in V$, we have $\mu(\{0\}) = 0$.

        \item\label{GeneralizedAbelianOperatorialRecurrence} For every unitary representation $\pi$ of $G$ and every vector $\xi \in H_\pi$, if $\sum_{v \in V}|\langle \pi(v)\xi,\xi\rangle|^p < \infty$ for some $p \in \mathbb{N}$, then $\xi$ is orthogonal to all eigenvectors of $\pi$. 

        \item\label{GeneralizedAbelianSpectralvdC} For any probability measure $\mu$ on $\widehat{G}$ satisfying $\sum_{v \in V}|\hat{\mu}(v)|^p < \infty$ for some $p \in \mathbb{N}$, we have that $\mu$ is continuous.
    \end{enumerate}
\end{theorem}

\begin{proof}
    The equivalence between \eqref{AbelianOperatorialRecurrence} and \eqref{GeneralizedAbelianOperatorialRecurrence} is a special case of the equivalence of \eqref{item:unitary_inv2} and \eqref{item:GeneralizedWMUnitary_inv2} in Theorem \ref{EquivalentCharacterizationsInCountableGroups}. To see that \eqref{AbelianSpectralvdC}$\rightarrow$\eqref{AbelianOperatorialRecurrence} and that \eqref{GeneralizedAbelianSpectralvdC}$\rightarrow$\eqref{GeneralizedAbelianOperatorialRecurrence}, it suffices to observe that the Spectral Theorem gives us a measure $\mu$ on $\widehat{G}$ for which $\hat{\mu}(g) = \langle \pi(g)\xi,\xi\rangle$ and $\mu(\{\chi\}) = ||P_{\chi}\xi||^2$, where $P_\chi:\mathcal{H}_\pi\rightarrow\mathcal{H}_\pi$ is the orthogonal projection onto the space of $\chi$-eigenvectors. To see that \eqref{AbelianOperatorialRecurrence}$\rightarrow$\eqref{AbelianSpectralvdC} and that \eqref{GeneralizedAbelianOperatorialRecurrence}$\rightarrow$\eqref{GeneralizedAbelianSpectralvdC}, it suffices to observe that the representation $\pi$ of $G$ on $L^2(\widehat{G},\mu)$ given by $(\pi(g)f)(\chi) = \chi(g)f(\chi)$ satisfies $\hat{\mu}(g) = \langle \pi(g)1,1\rangle$ and $\mu(\{\chi\}) = ||P_{\chi}1||^2$.
\end{proof}

In the work of Rodr\'iguez \cite{SaulsvdC}, a subset $V$ of a countably infinite group $G$ is a \textbf{vdC set} if for any measure preserving system $(X,\mathscr{B},\mu,(\tau_g)_{g \in G})$ and any $f \in L^\infty(X,\mu)$ satisfying $\langle \tau_vf,f\rangle = 0$ for all $v \in V$, we have $\int_Xfd\mu = 0$.\footnote{His decision to take this as the definition of vdC set was motivated by the fact that for an amenable group $G$, this definition coincides with the analogue of Definition \ref{vdCSetDefinition}.} Theorem \ref{EquivalentCharacterizationsInCountableGroups} shows us that every set of operatorial recurrence is a vdC set and Theorem \ref{FvdCIsvdCForAmenableGroups} shows us that vdC sets are sets of operatorial recurrence if $G$ is abelian. If Question \ref{MainConjecture} is answered in the positive, then every vdC set will also be a set of operatorial recurrence in any countably infinite group $G$.

Our next result is a list of properties of sets of operatorial recurrence, and this list is essentially the same list of properties of vdC sets given in \cite[Section 5]{SaulsvdC}. We only give the proof of one of these results here since the proofs of the rest are nearly identical to the analogous results for vdC sets.

\begin{theorem}\label{PropertiesOfSetsOfOperatorialRecurrence}
    Let $G$ be a countably infinite group and let $V \subseteq G$ be a set of operatorial recurrence.
    \begin{enumerate}[(i)]
        \item If $V = V_1\cup V_2$, then one of $V_1$ and $V_2$ is a set of operatorial recurrence.

        \item If $\phi:G\rightarrow H$ is a group homomorphism, then $\phi(V)$ is a set of operatorial recurrence.

        \item There exist sets of operatorial recurrence $V_1,V_2 \subseteq V$ with $V_1\cap V_2 = \emptyset$.

        \item If $L$ is a group containing $G$ as a subgroup, then $V$ is a set of operatorial recurrence in $L$.
        
        \item If $H$ is a subgroup of $G$ and $V \subseteq H$, then $V$ is a set of operatorial recurrence in $H$. 
        
        \item $V^{-1} := \{v^{-1}\ |\ v \in V\}$ is a set of operatorial recurrence in $G$.

        \item If $A \subseteq G$ is infinite, then $V := \{ab^{-1}\ |\ a,b \in A\}$ is a set of operatorial recurrence.\footnote{Kamae and Mendes-France \cite{KMFvdC} showed that if $I \subseteq \mathbb{N}$ is infinite, then $\{n-m\ |\ m,n \in I, m < n\}$ is a vdC set in $\mathbb{N}$.} Similarly, if $A \subseteq G$ is \textbf{thick}, i.e., for any finite set $H \subseteq G$ there exists $g_H \in G$ for which $g_HH \subseteq A$, then $A$ is a set of operatorial recurrence.

        \item If $H$ is a finite index subgroup of $G$, then $G\setminus H$ is not a set of operatorial recurrence. Similarly, if $H \subseteq G\setminus\{e\}$ is a finite set, then $H$ is not a set of operatorial recurrence.
    \end{enumerate}
\end{theorem}

\begin{proof}
    The only part of this Theorem whose proof is different from the analogous statement in \cite[Section 5]{SaulsvdC} is part (v). In particular, we need to show that if $H$ is a subgroup of $G$, and $V \subseteq H$, then $V$ is a set of operatorial recurrence in $H$. 
    Let $\pi$ be a representation of $H$ on $\mathcal{H}$, let $\mathcal{H} = \mathcal{H}_1\oplus\mathcal{H}_2$ where $\mathcal{H}_1$ is the space of $\pi$-invariant vectors, let $\pi'$ on $\mathcal{H}_2\otimes \ell^2(G/H)$ be the induced representation from $H$ to $G$ of $\pi$ restricted to $\mathcal{H}_2$, and let $\kappa$ be the direct sum of the trivial representation of $G$ on $\mathcal{H}_1$ with $\pi'$. Now let $\xi \in \mathcal{H}$ be such that $\langle \pi(v)\xi,\xi\rangle = 0$ for all $v \in V$. Let $\xi = \xi^{(1)}+\xi^{(2)}$ with $\xi^{(i)} \in \mathcal{H}_i$ and let $\xi' \in \mathcal{H}_1\oplus\left(\mathcal{H}_2\otimes\ell^2(G/H)\right)$ be given by $\xi' = \xi^{(1)}+\xi^{(2)}\otimes\mathbbm{1}_{\{eH\}}$. We see that $\langle \kappa(v)\xi',\xi'\rangle = 0$ for all $v \in V$, so $\xi^{(1)} = 0$ since $\mathcal{H}_1$ is the space of $\kappa$-invariant vectors, which yields the desired result.
\end{proof}
\bibliographystyle{abbrv}
\begin{center}
	\bibliography{references}

@book {AlgebraInTheSCC,
    AUTHOR = {Hindman, Neil and Strauss, Dona},
     TITLE = {Algebra in the {S}tone-\v{C}ech compactification: Theory and applications},
    SERIES = {De Gruyter Textbook},
   EDITION = {Second revised and extended},
 PUBLISHER = {Walter de Gruyter \& Co., Berlin},
      YEAR = {2012},
     PAGES = {xviii+591},
      ISBN = {978-3-11-025623-9},
   MRCLASS = {54-02 (03E05 22A15 54D35 54H99)},
  MRNUMBER = {2893605},
}

@phdthesis{SohailsPhDThesis,
    title    = {Topics in ergodic theory and ramsey theory},
    school   = {the Ohio State University},
    author   = {Farhangi, Sohail},
    year     = {2022},
    type     = {{PhD} dissertation},
}

@book {ErgodicTheoryViaJoinings,
    AUTHOR = {Glasner, Eli},
     TITLE = {Ergodic theory via joinings},
    SERIES = {Mathematical Surveys and Monographs},
    VOLUME = {101},
 PUBLISHER = {American Mathematical Society, Providence, RI},
      YEAR = {2003},
     PAGES = {xii+384},
      ISBN = {0-8218-3372-3},
   MRCLASS = {37A15 (28Dxx 37A25 37A35 37A45 37B99 54H20)},
  MRNUMBER = {1958753},
MRREVIEWER = {Andr\'{e}s\ del\ Junco},
       DOI = {10.1090/surv/101},
       URL = {https://doi.org/10.1090/surv/101},
}

@book {ATempelmanBook,
    AUTHOR = {Tempelman, Arkady},
     TITLE = {Ergodic theorems for group actions},
    SERIES = {Mathematics and its Applications},
    VOLUME = {78},
      NOTE = {Informational and thermodynamical aspects,
              Translated and revised from the 1986 Russian original},
 PUBLISHER = {Kluwer Academic Publishers Group, Dordrecht},
      YEAR = {1992},
     PAGES = {xviii+399},
      ISBN = {0-7923-1717-3},
   MRCLASS = {22D40 (28D15 47A35 47D03 60G60)},
  MRNUMBER = {1172319},
MRREVIEWER = {Arlan Ramsay},
       DOI = {10.1007/978-94-017-1460-0},
       URL = {https://doi-1org-102eb17rn1080.han.amu.edu.pl/10.1007/978-94-017-1460-0},
}

@incollection {RuzsavanderCorput,
    AUTHOR = {Ruzsa, I. Z.},
     TITLE = {Connections between the uniform distribution of a sequence and
              its differences},
 BOOKTITLE = {Topics in classical number theory, {V}ol. {I}, {II}
              ({B}udapest, 1981)},
    SERIES = {Colloq. Math. Soc. J\'{a}nos Bolyai},
    VOLUME = {34},
     PAGES = {1419--1443},
 PUBLISHER = {North-Holland, Amsterdam},
      YEAR = {1984},
   MRCLASS = {11K06},
  MRNUMBER = {781190},
MRREVIEWER = {Gutti J. Babu},
}

@article {vanderCorputSetsInZ^d,
    AUTHOR = {Bergelson, Vitaly and Lesigne, Emmanuel},
     TITLE = {Van der {C}orput sets in {$\Bbb Z^d$}},
   JOURNAL = {Colloq. Math.},
  FJOURNAL = {Colloquium Mathematicum},
    VOLUME = {110},
      YEAR = {2008},
    NUMBER = {1},
     PAGES = {1--49},
      ISSN = {0010-1354},
   MRCLASS = {11K06 (28D05 37A45 42A82)},
  MRNUMBER = {2353898},
MRREVIEWER = {Bryna Kra},
       DOI = {10.4064/cm110-1-1},
       URL = {https://doi-org.proxy.lib.ohio-state.edu/10.4064/cm110-1-1},
}

@article {KroneckerImpliesGaussian,
    AUTHOR = {Foia\c{s}, Ciprian and Str\u{a}til\u{a}, \c{S}erban},
     TITLE = {Ensembles de {K}ronecker dans la th\'{e}orie ergodique},
   JOURNAL = {C. R. Acad. Sci. Paris S\'{e}r. A-B},
  FJOURNAL = {Comptes Rendus Hebdomadaires des S\'{e}ances de l'Acad\'{e}mie des
              Sciences. S\'{e}ries A et B},
    VOLUME = {267},
      YEAR = {1968},
     PAGES = {A166--A168},
      ISSN = {0151-0509},
   MRCLASS = {28.70 (60.00)},
  MRNUMBER = {232911},
MRREVIEWER = {E. S. Boylan},
}

@book {CornfeldFominSinai,
    AUTHOR = {Cornfeld, I. P. and Fomin, S. V. and Sina\u{\i}, Ya. G.},
     TITLE = {Ergodic theory},
    SERIES = {Grundlehren der mathematischen Wissenschaften [Fundamental
              Principles of Mathematical Sciences]},
    VOLUME = {245},
      NOTE = {Translated from the Russian by A. B. Sosinski\u{\i}},
 PUBLISHER = {Springer-Verlag, New York},
      YEAR = {1982},
     PAGES = {x+486},
      ISBN = {0-387-90580-4},
   MRCLASS = {28D05 (54H20 58F11)},
  MRNUMBER = {832433},
MRREVIEWER = {D. Newton},
       DOI = {10.1007/978-1-4615-6927-5},
       URL = {https://doi-1org-102eb1719021b.han.amu.edu.pl/10.1007/978-1-4615-6927-5},
}

@article {TilingAmenableGroups,
    AUTHOR = {Downarowicz, Tomasz and Huczek, Dawid and Zhang, Guohua},
     TITLE = {Tilings of amenable groups},
   JOURNAL = {J. Reine Angew. Math.},
  FJOURNAL = {Journal f\"{u}r die Reine und Angewandte Mathematik. [Crelle's
              Journal]},
    VOLUME = {747},
      YEAR = {2019},
     PAGES = {277--298},
      ISSN = {0075-4102,1435-5345},
   MRCLASS = {37B52 (20F65 37B05)},
  MRNUMBER = {3905135},
MRREVIEWER = {Dou\ Dou},
       DOI = {10.1515/crelle-2016-0025},
       URL = {https://doi.org/10.1515/crelle-2016-0025},
}

@article {vdCStrongerThanRecurrence,
    AUTHOR = {Bourgain, J.},
     TITLE = {Ruzsa's problem on sets of recurrence},
   JOURNAL = {Israel J. Math.},
  FJOURNAL = {Israel Journal of Mathematics},
    VOLUME = {59},
      YEAR = {1987},
    NUMBER = {2},
     PAGES = {150--166},
      ISSN = {0021-2172},
   MRCLASS = {11B75 (11B83 11K31)},
  MRNUMBER = {920079},
MRREVIEWER = {Zun Shan},
       DOI = {10.1007/BF02787258},
       URL = {https://doi-1org-102eb17dk008b.han.amu.edu.pl/10.1007/BF02787258},
}

@article {RefiningBourgain,
    AUTHOR = {Mountakis, Andreas},
     TITLE = {Distinguishing sets of strong recurrence from van der Corput Sets},
   JOURNAL = {Israel J. Math.},
  FJOURNAL = {Israel Journal of Mathematics},
    VOLUME = {},
      YEAR = {2024},
    NUMBER = {},
     PAGES = {},
      ISSN = {},
   MRCLASS = {},
  MRNUMBER = {},
MRREVIEWER = {},
       DOI = {},
       URL = {},
}

@article {KMFvdC,
    AUTHOR = {Kamae, T. and Mend\`es France, M.},
     TITLE = {van der {C}orput's difference theorem},
   JOURNAL = {Israel J. Math.},
  FJOURNAL = {Israel Journal of Mathematics},
    VOLUME = {31},
      YEAR = {1978},
    NUMBER = {3-4},
     PAGES = {335--342},
      ISSN = {0021-2172},
   MRCLASS = {10K05 (10L02)},
  MRNUMBER = {516154},
MRREVIEWER = {Marthe Grandet},
       DOI = {10.1007/BF02761498},
       URL = {https://doi-1org-102eb17dk008b.han.amu.edu.pl/10.1007/BF02761498},
}

@article {vdCIsOperatorialRecurrent,
    AUTHOR = {Nin\v{c}evi\'{c}, Marina and Rabar, Braslav and Slijep\v{c}evi\'{c}, Sini\v{s}a},
     TITLE = {Ergodic characterization of van der {C}orput sets},
   JOURNAL = {Arch. Math. (Basel)},
  FJOURNAL = {Archiv der Mathematik},
    VOLUME = {98},
      YEAR = {2012},
    NUMBER = {4},
     PAGES = {355--360},
      ISSN = {0003-889X},
   MRCLASS = {37A45 (11K36)},
  MRNUMBER = {2914351},
MRREVIEWER = {Tom Sanders},
       DOI = {10.1007/s00013-012-0370-6},
       URL = {https://doi-1org-102eb17kg003c.han.amu.edu.pl/10.1007/s00013-012-0370-6},
}

@article {StationarySequencesOfVectorsvdC,
    AUTHOR = {Alon, N. and Peres, Y.},
     TITLE = {A note on {E}uclidean {R}amsey theory and a construction of
              {B}ourgain},
   JOURNAL = {Acta Math. Hungar.},
  FJOURNAL = {Acta Mathematica Hungarica},
    VOLUME = {57},
      YEAR = {1991},
    NUMBER = {1-2},
     PAGES = {61--64},
      ISSN = {0236-5294},
   MRCLASS = {11B05 (05D10 46C99)},
  MRNUMBER = {1128841},
MRREVIEWER = {F. Schweiger},
       DOI = {10.1007/BF01903803},
       URL = {https://doi-1org-102eb170j2550.han.amu.edu.pl/10.1007/BF01903803},
}

@article {BanachLimitsvdC,
    AUTHOR = {Peres, Yuval},
     TITLE = {Application of {B}anach limits to the study of sets of
              integers},
   JOURNAL = {Israel J. Math.},
  FJOURNAL = {Israel Journal of Mathematics},
    VOLUME = {62},
      YEAR = {1988},
    NUMBER = {1},
     PAGES = {17--31},
      ISSN = {0021-2172},
   MRCLASS = {11K06},
  MRNUMBER = {947826},
       DOI = {10.1007/BF02767350},
       URL = {https://doi-1org-102eb170j2568.han.amu.edu.pl/10.1007/BF02767350},
}

@article {SaulsvdC,
    AUTHOR = {Rodr\'iguez Mart\'in, Sa\'ul},
     TITLE = {An inverse of {F}urstenberg’s correspondence
principle and applications to van der {C}orput sets},
   JOURNAL = {Trans. Am. Math. Soc.},
  FJOURNAL = {Transactions of the American Mathematical Society},
    VOLUME = {},
      YEAR = {To appear},
     PAGES = {},
      ISSN = {},
   MRCLASS = {},
  MRNUMBER = {},
       DOI = {},
       URL = {},
}

@article{avigad2012inverting,
  title={Inverting the Furstenberg correspondence},
  author={Avigad, Jeremy},
  journal={Discrete and Continuous Dynamical Systems},
  volume={32},
  number={10},
  pages={3421--3431},
  year={2012},
  publisher={Discrete and Continuous Dynamical Systems}
}

@article{InverseOfFurstenbergsCorrespondenceForN,
  title={An inverse of Furstenberg’s correspondence principle
and applications to nice recurrence},
  author={Fish, Alexander and Skinner, Sean},
  journal={Discrete Contin. Dyn. Syst.},
  volume={},
  number={},
  pages={},
  year={To appear},
  publisher={arxiv}
}

@book {AmenabilityByPaterson,
    AUTHOR = {Paterson, Alan L. T.},
     TITLE = {Amenability},
    SERIES = {Mathematical Surveys and Monographs},
    VOLUME = {29},
 PUBLISHER = {American Mathematical Society, Providence, RI},
      YEAR = {1988},
     PAGES = {xx+452},
      ISBN = {0-8218-1529-6},
   MRCLASS = {43-02 (22-02 43A07 46Lxx)},
  MRNUMBER = {961261},
MRREVIEWER = {C. Chou},
       DOI = {10.1090/surv/029},
       URL = {https://doi-1org-102eb17k00086.han.amu.edu.pl/10.1090/surv/029},
}

@article{APFunctionsAndRepresentations,
 author = {Shtern, A. I.},
 title = {Almost periodic functions and representations in locally convex spaces},
 fjournal = {Russian Mathematical Surveys},
 journal = {Russ. Math. Surv.},
 issn = {0036-0279},
 volume = {60},
 number = {3},
 pages = {489--557},
 year = {2005},
 language = {English},
 doi = {10.1070/RM2005v060n03ABEH000849},
 zbMATH = {5063159},
 Zbl = {1119.43005}
}

@misc{GreenleafBook,
 author = {Greenleaf, Frederick P.},
 title = {Invariant means on topological groups and their applications. ({Invariantnye} srednie na topologiceskih gruppah i ih prilozenija.) {\"U}bersetzung aus dem {Englischen} von {V}. {F}. {Pahomov}. {Herausgegeben} von {Ja}. {G}. {Sinai}},
 year = {1973},
 language = {Russian},
 howpublished = {Biblioteka sbornika ''matematika''. {Moskau}: {Verlag} ''{Mir}''. 136 {S}. {R}. 0.63 (1973).},
 zbMATH = {3398130},
 Zbl = {0252.43005}
}
\end{center}

\end{document}